\documentclass{siamart251216}

\usepackage{amsfonts}
\usepackage[T1]{fontenc}
\usepackage[utf8]{inputenc}
\usepackage{graphicx}
\usepackage{epstopdf}
\usepackage{amssymb}
\usepackage[all]{xy}
\usepackage{hyperref}
\usepackage{enumitem}
\usepackage{mdframed}
\usepackage{esint}
\usepackage{bbm}
\usepackage{comment}
\usepackage{mathdots}
\usepackage[skip=1pt]{caption}
\usepackage{subcaption}
\usepackage{algorithm}
\usepackage{algpseudocode}
\usepackage{svg}
\usepackage{float}
\restylefloat{figure}
\ifpdf
  \DeclareGraphicsExtensions{.eps,.pdf,.png,.jpg}
\else
  \DeclareGraphicsExtensions{.eps}
\fi

\newsiamremark{remark}{Remark}
\newsiamremark{hypothesis}{Hypothesis}
\crefname{hypothesis}{Hypothesis}{Hypotheses}
\newsiamthm{claim}{Claim}
\newsiamremark{fact}{Fact}
\crefname{fact}{Fact}{Facts}

\headers{Linear Recurrent Neural Networks as Time-Delay Embeddings}{F. Ng and J. N. Kutz}

\title{Linear Recurrent Neural Networks as Time-Delay Embeddings\thanks{Posted 26 May 2026.
\funding{\textbf{This work was funded by the Air Force Office of Scientific Research (FA9550-24-1-0141).}}}}

\author{Fisher Ng \thanks{Department of Applied Mathematics, University of Washington, Seattle, WA 
  (\email{fisherng@uw.edu}, \url{website}).}
\and J. Nathan Kutz \thanks{Autodesk Research, London, UK 
  (\email{nathan.kutz@autodesk.com}).}}

\usepackage{amsopn}
\DeclareMathOperator{\diag}{diag}

\newcommand{\R}{\mathbb{R}}
\newcommand{\C}{\mathbb{C}}

\newcommand{\bb}{\mathbf{b}}

\newcommand{\bw}{\mathbf{w}}

\newcommand{\bx}{\mathbf{x}}

\newcommand{\by}{\mathbf{y}}
\newcommand{\bh}{\mathbf{h}}

\newcommand{\bA}{\mathbf{A}}
\newcommand{\bB}{\mathbf{B}}
\newcommand{\bC}{\mathbf{C}}

\newcommand{\bH}{\mathbf{H}}
\newcommand{\bI}{\mathbf{I}}

\newcommand{\bM}{\mathbf{M}}

\newcommand{\bQ}{\mathbf{Q}}

\newcommand{\bP}{\mathbf{P}}
\newcommand{\bT}{\mathbf{T}}
\newcommand{\bU}{\mathbf{U}}
\newcommand{\bV}{\mathbf{V}}
\newcommand{\bW}{\mathbf{W}}

\newcommand{\bSigma}{\boldsymbol{\Sigma}}

\newcommand{\bLambda}{\boldsymbol{\Lambda}}

\newcommand{\op}{\mathrm{op}}
\newcommand{\bzero}{\mathbf{0}}

\newcommand{\cA}{\mathcal{A}}

\newcommand{\cO}{\mathcal{O}}

\newcommand{\A}{\mathbb{A}}
\newcommand{\D}{\mathbb{D}}
\newcommand{\M}{\mathbb{M}}
\newcommand{\T}{\mathbb{T}}
\newcommand{\U}{\mathbb{U}}
\newcommand{\V}{\mathbb{V}}

\begin{document}

\maketitle

\begin{abstract}
    Sequence models, and particularly Linear Recurrent Neural Networks (LRNNs) of the form $\bh_{k+1} = \bW \bh_{k} + \by_k + \bb$, are widely applicable in time-series analysis for dynamical systems, yet, as black-box algorithms, much is unknown about why they perform well. 
    In this work, we leverage Takens' embedding theorem, which provides conditions under which partially observed time series organized into delay-coordinate vectors can faithfully represent the original system's dynamics, as a theoretical framework for explaining how and why sequence models preserve and reconstruct dynamical systems.
    For LRNNs, concatenating output states into delay-coordinate vectors gives rise to a ``delay" matrix $\M_{n,m}\in \C^{(nm) \times (n+1)m}$: a block matrix consisting of identity matrices $\bI \in \R^{m \times m}$ repeated $n$ times along the main diagonal and weight matrices $\bW \in \C^{m \times m}$ featured $n$ times along the super-diagonal. 
    $\M_{n,m}$ relates the delay-coordinates of the input time series to those of the LRNN output states, and, for $\M_{n,m}$ to be an embedding, it must be full row-rank.
    We provide explicit conditions for $\M_{n,m}$ to be full row-rank and prove the condition number of $\M_{n,m}$ and determinant of $\M_{n,m} \M_{n,m}^*$--measures of embedding stability--are bounded independent of $n$, at least for certain ranges of $\bW$'s singular values: namely, when $\sigma_{\max}(\bW) \le 1$. 
    This result explains why the spectrum of $\bW$ for trained LRNNs tends to converge to within the unit circle.
\end{abstract}



\section{Introduction}
In many data-driven applications, the equations that govern the dynamics of the system under study are unknown, and the goal is to infer the unobserved states or to discover a model for the dynamics of the system based on a set of partial observations in the form of time series data ~\cite{Crutchfield_1987, Broomhead_1989, Tong_1990, Bakarji_2022}.
More precisely, we consider a dynamical system
\begin{align}\label{eq:Dynamical_System}
    \dot{\bx}(t) = \mathbf{f}(\bx(t)) \quad \mbox{where} \quad \bx(t) \in \R^p
\end{align}
where $\mathbf{f}$ is the vector field and $\bx(t)$ is the state of the system.
When the system is dissipative, it loses energy over time, and, in the long-term, its solution can settle to and evolve on a compact, low-dimensional invariant set $\cA \subseteq \R^p$ called an \textit{attractor}, of fractal dimension $d$.
The collected time series $\{\by_k\}_{k = 0}^T \in \R^m$ for some $T \in \mathbb{N}$ is modeled by a smooth, multivariate function $\mu: \R^{p} \to \R^m$, $m \le p$, that takes measurements:
\begin{align}\label{eq:System_Observations}
    \by_{k} := \by(t - k \tau) = \mu(\bx(t - k \tau)) \in \R^m
\end{align}
where $\tau > 0$ is a chosen time delay.

Takens' embedding\footnote{In differential topology, an embedding is a one-to-one map the derivative of which is also one-to-one, which may be slightly different than the notion of an embedding in machine learning literature.} theorem~\cite{Takens_1981} is a foundational result in dynamical systems theory that provides a rigorous method for characterizing the attractor of a dynamical system from a single observed time series. 
The theorem states that the observation of a scalar measurement (i.e. $\mu: \R^p \to \R$) of a dynamical system over time can reconstruct a topologically equivalent representation of the original system's attractor by constructing delay-coordinate vectors — that is, by forming vectors of the form 
\begin{align}\label{eq:Delay_Coordinate_Map}
    \Phi(\bx(t)) = \begin{bmatrix}\by(t) & \by(t-\tau) & \by(t-2\tau) & \cdots & \by(t-(n-1)\tau)\end{bmatrix}^\top \in \R^{n}
\end{align}
where $n$ is the embedding dimension. 
Provided $n$ is large enough--specifically, $n \geq 2d+1$, where $d$ is the dimension of the attractor, along with a few other mild assumptions on $\tau$ in relation to the periodic orbits in the attractor, then with probability one the reconstructed attractor is diffeomorphically equivalent to the original--meaning it preserves the geometric and topological properties of the true attractor.
The theorem relies on the fact that, when the dynamics of the state variables of a system are coupled, information about unobserved states is implicitly-available in the explicitly-observed time series data, and with enough measurements over time, it becomes possible to infer the evolution of the full state.
This theorem is remarkably powerful since it implies that a hidden, high-dimensional system can be faithfully characterized from a single observable signal, making Takens' theorem a theoretical underpinning of time-series analysis in nonlinear and chaotic dynamics ~\cite{Crutchfield_1987, Broomhead_1989, Tong_1990, Bakarji_2022}.

While for many simple systems, collecting time series observations of a single state variable is sufficient (i.e. $\mu: \R^p \to \R$), for more complex systems, drawing on the time-lagged observations of multiple state variables (i.e. $\mu: \R^p \to \R^m$), organized into multivariate time series, can enable more robust reconstructions ~\cite{Cao_1998} and be useful for multi-scale modeling ~\cite{Judd_1998, Hirata_2006, Gao_2021, Tan_2023}. 
Takens' Embedding theorem has been extended to multivariate delay-coordinate maps to prove that, under the appropriate conditions, such maps can also embed the dynamics of a system on its attractor ~\cite{SYC_1991, Deyle_2011}.

Although Takens' embedding theorem provides a means for forming an image of a system's attractor, the image is only diffeomorphic to the original attractor, making reconstruction impractical. 
Sequence models, such as recurrent neural networks (RNNs) ~\cite{Elman_1990}, have proven useful for time-series analysis by finding better representations of the underlying time series data ~\cite{Bakker_2000, Uribarri_2022, Bhat_2022_RNN} and for reconstructing the original dynamics when using training data of the original attractor as a reference ~\cite{Ebers_SHRED_2023, Williams_2024_SHRED, Faraji_SHRED_2025, Tomasetto_SHRED_2025, Gao_SHRED_2025}. 
Because of the black-box nature of nonlinear sequence models, understanding the conditions under which they work is unknown aside from empirical methods such as training losses, cross-validation, and performance metrics.
Thus, linear RNNs (LRNNs) have received more attention due to their interpretability and how they lend themselves well to analysis  ~\cite{Stolzenburg_2025, Jarne_2023, Jarne_2024}.
Specifically, efforts have contributed toward this goal by analyzing the stability of RNNs ~\cite{Zucchet_2024} and particularly their weight matrices, using random matrix methods, in order to suggest stability when eigenvalues are contained in the unit circle ~\cite{Glorot_2010, Orvieto_2023, Gupta_2022, Bar_2025, Rajan_2006}.

From an analysis point of view, the simplest tractable sequence model to consider is a first-order recurrence relationship of the form:
\begin{align}\label{eq:LRNN}
    \bh_{k+1} = \bW \bh_k + \by_k + \bb
\end{align}
where $\bW \in \R^{m \times m}$ is a ``recurrence" or ``weight" matrix; $\bh_k \in \R^m$ is a system state--initialized with some initial condition $\bh_0 \in \R^m$; $\by_k \in \R^m$ is an external input or driving term; and $\bb \in \R^m$ is a constant or ``bias" term.
Such difference equations also arise naturally in diverse applications, such as the age distribution of segments of a growing population over time ~\cite{Leslie_1945, Lewis_1942}, economic growth models ~\cite{Leontief_1951, Bellman_1952, Merton_1973}, and the page rank algorithm used for Internet searches using Markov chain models ~\cite{Brin_1998}.
Regarding the asymptotic stability of \eqref{eq:LRNN}, $\bh_k$ converges asymptotically to a fixed state $\bh_k^*$ provided the absolute maximal eigenvalue $\max_{k \in \{1,...,m\}}|\lambda_k(\bW)| < 1$ ~\cite{Elaydi_2005} when there is no external time series forcing term $\by_k$.

More broadly, the nature of temporal sequential data, organized into a matrix structure based on delay-coordinates--as shown in the next subsection, provides a mathematical framework for relating Takens' embedding theorem to time-series analysis and recurrent neural networks.  
Of particular interest are the conditions under which the output of the LRNN processed based on the input time series preserves information and allows recovery of the measured system.
Thus, understanding the properties of the delay matrix (see \eqref{eq:LRNN_Delay_Coordinates}) is critical to understanding when the act of processing input time series using a LRNN outputs a representation that preserves the image of the original dynamical system's attractor.
The significant difference in the stability analysis of the linear recurrence relation \eqref{eq:LRNN} undertaken in this work is that it works with collections of delay-coordinates in the context of Takens' embedding theorem, rather than a typical iterate-to-iterate analysis.
Additionally, the results in this paper offer explicit, deterministic conditions for LRNN output states to form embeddings, this as a complement to the random matrix methods or empirical approaches often in machine learning literature.

\subsection{Building a Theoretical Framework}

Suppose that time series observations $\{\by_{k}\}_{k = 0}^T$ of the dynamical system \eqref{eq:Dynamical_System} are organized into a multivariate delay-coordinate map of length $n < T$, and further that the observations are input into the first-order difference equation \eqref{eq:LRNN}.
The result will be $n + 1$ simultaneous equations of the LRNN states $\{\bh_{k-\ell}\}_{\ell = 0}^{n + 1}$. 
Concatenating the states into a delay-coordinate representation, and re-arranging gives rise to the system of equations:
\begin{align} \label{eq:LRNN_Delay_Coordinates}
    \begin{bmatrix}
        \bI & - \bW & \bzero & \cdots & \bzero \\
        \bzero & \bI & - \bW & \ddots & \bzero \\
        \vdots & \ddots & \ddots & \ddots & \vdots \\
        \bzero & \bzero & \cdots & \bI & -\bW
    \end{bmatrix} 
    \begin{bmatrix}
        \bh_{k} \\
        \bh_{k - 1} \\
        \vdots \\
        \bh_{k - n} \\
        \bh_{k - (n + 1)}
    \end{bmatrix}  
    = \begin{bmatrix}
        \by_{k-1} \\ \by_{k-2} \\ \vdots \\ \by_{k - (n+1)} 
    \end{bmatrix} + \begin{bmatrix}
        \bb \\ \bb \\ \vdots \\ \bb 
    \end{bmatrix}.
\end{align}
Note that the vector of observations
\begin{align}
    \Phi(\bx_{k}) = \begin{bmatrix}
        \by_k & \by_{k-1} & \cdots & \by_{k - n}
    \end{bmatrix}^\top \in \R^{mn}
\end{align}
is exactly a multivariate delay-coordinate map representing $\bx_{k}$.
By shifting the iterate to $k-1$, the right-hand-side vector of observations is in fact $\Phi(\bx_{k-1})$.
Because an embedding remains an embedding when shifted in the reconstruction space by a constant, time-independent vector, $\Phi(\bx_{k-1})$ is still an embedding when perturbed by the block vector $\boldsymbol{1} \otimes \bb$, where $\otimes$ denotes the Kronecker product and $\boldsymbol{1} \in \R^n$.
Likewise, the vector of concatenated recurrence states, denoted by:
\begin{align}\label{eq:h_Delay_Coordinate_Map}
    \Psi(\bh_{k}) = \begin{bmatrix}
        \bh_k & \bh_{k-1} & \cdots & \bh_{k - (n+1)}
    \end{bmatrix}^\top \in \R^{m(n+1)}
\end{align}
is a sort of delay-coordinate map as well, albeit slightly different than that of Takens'.
Finally, by defining the bi-diagonal block matrix on the left-hand-side as $\M_{n,m}$, we can solve for the delay-coordinate representation of the latent states $\Psi(\bh_k)$ in terms of the time-delay coordinates of the system by applying the pseudo-inverse of $\M_{n,m}$:
\begin{align} \label{eq:LRNN_Delay_Coordinates_Sol}
    \Psi(\bh_{k}) = \M_{n,m}^\dagger (\Phi(\bx_{k-1}) + \boldsymbol{1} \otimes \bb)
\end{align}
To ensure that such a solution is well-behaved, in the sense that the mapping preserves information in the input time series and ensures the original dynamics are recoverable, we would like to know the properties of $\M_{n,m}$, mainly the distribution of its singular values so that we can determine: (i) its rank, to ensure that the mapping preserves information; (ii) its condition number, to quantify the stability of the embedding map under the recurrence relation; and (iii) the general structure of its pseudo-inverse $\M_{n,m}^\dagger$, to see how the solutions $\Psi(\bh_k)$ draw on and combine the relevant input data $\{\by_{k-\ell}\}_{\ell = 1}^{n+1}$. 
Knowing these properties of the singular values of $\M_{n,m}$ would allow the identification of appropriate conditions to impose on the weights $\bW$ for LRNNs when applying them to time series data.
It is therefore the goal of this paper to explicitly determine the properties of $\M_{n,m}$ to provide a basis for analyzing LRNNs as mappings that are a composition of time-delay embeddings.

\subsection{Main Contributions}

Our aim in this work is to establish Takens' time-delay embedding theory as a theoretical framework for explaining why LRNNs as sequence models are well-suited to time-series analysis and state space reconstruction, and in particular how they can preserve the information about states and their evolution over time.
Not only so, this work arrives at explicit, deterministic bounds under which input time series data processed using LRNNs are guaranteed to be diffeomorphically-equivalent to the attractor of the original system \eqref{eq:Dynamical_System}.
Specifically, our analysis shows the following:\\

\begin{enumerate}
    \item The outputs of sequence models are in fact composite embeddings: the sequence model must be an embedding of an already-valid time-delay embedding.
    This has fundamental consequences for the input data, which must form an embedding: the number of states observed and the observation function $\mu$ used, the number of time-lagged observations available and the choice of sampling frequency $\tau$, and the dimension of the attractor and the nature of its periodic orbits all become essential to know for effective time series modeling.
    For the LRNN to be an embedding, the nature of the network architecture: the size of the hidden state, the behavior of the weights, and the choice of activation function--becomes equally essential for time-delay embedding.
    \item For the case of scalar LRNNs $h_{k+1} = \omega h_k + x_k + b$, we show that for any $\omega \in \C$, $\mathrm{rank}(\bM_n) = n$, guaranteeing that any choice of weight parameter $\omega$ will mean $\Psi(h_k)$ \eqref{eq:h_Delay_Coordinate_Map} is a time-delay embedding of the state $\bx(t)$, provided $\Phi(\bx(t))$ itself is an embedding. 
    Although all $\omega \in \C$ are technically admissible choices, weight parameters within the unit circle (i.e. $\omega \in \{\omega \in \C : |\omega| < 1\})$ ensure that $\bM_n$ has a bounded condition number and a small generalized determinant (see \eqref{eq:Generalized_Determinant}).
    This agrees with the observation that the eigenvalue distributions of RNN weights tend to settle within the unit circle. 
    Importantly, the bound for the condition number and determinant scale independently of the number of lags $n$ used in the delay-coordinate map, and, although increasing $n$ does increase the ill-conditioning of $\bM_n$, it only does so up to a finite limit, potentially allowing for ``infinite" look back.
    \item For the case of LRNNs as in \eqref{eq:LRNN} equipped with Hermitian weight matrices $\bW$, we show that for any $\bW \in \C^{m \times m}$, $\mathrm{rank}(\M_{n,m}) = mn$, and thus the LRNN is always an embedding from $\Phi(\bx(t))$ to $\Psi(\bh_k)$. 
    Again, although $\M_{n,m}$ is generally an embedding of $\cA$, $\M_{n,m}$ has a low condition number and small generalized determinant when $\sigma_{\max}(\bW) < 1$, and thus is a more stable embedding when all eigenvalues of $\bW$ lie within the unit circle.
    The bounds on the condition number and determinant are again independent of both $m$ and $n$, implying that choosing high-dimensional latent spaces and long delay-coordinate maps will not radically destabilize the embedding.
    \item For the case of LRNNs \eqref{eq:LRNN} equipped with arbitrary $\bW$, we prove that a sufficient condition\footnote{See Theorem \ref{thm:Block_Mat_Gen_Diag_Dom} for a better but more complicated condition.} for $\M_{n,m}$ to be rank $nm$, and thus preserve all information in $\Phi(\bx(t))$ when mapping to coordinates $\Psi(\bh_k)$, is that $\sigma_{\max}(\bW) < \frac{1}{2}(1 + \sigma_{\min}(\bW)^2)$. 
    This condition is best satisfied when the spectrum of $\bW$ is contained well-within the unit circle. 
    Notably, we also prove that LRNNs with unitary weight matrices $\bW$ are also embeddings.
    \item The above results are \textit{deterministic} bounds and apply to \textit{arbitrary} weight matrices. 
    They are consistent with empirical approaches to training RNNs and observations about the conditions that promote their stability. 
    Thus, the above insights are of both pedagogical value in that they provide theoretical insights to complement existing empirical observations, and of practical value in their ability to identify conditions to impose on weight matrices $\bW$ and on the input data to ensure topological and differential information of the time series is preserved when processed using RNNs.
    It also sets the stage for our future work, in which we will prove the conditions under which nonlinear RNNs are time-delay embeddings of time series data.
\end{enumerate}

\section{Preliminaries}

In this section, we make explicit the conventions and notation, as well as definitions and theorems, used throughout the paper.
Other results, used in specific sections, are introduced as needed.

While often in applications the LRNNs of interest are real-valued, we consider the generalized setting in which they are complex for completeness. 
We reserve bold symbols (ex. $\bI, \bzero, \bW, \bU, \bSigma, \bV$, etc.) for matrices and blackboard bold symbols (ex. $\M_{n,m}, \T$, etc., excluding $\R$ and $\C$) for block matrices. 
For a complex number $\omega \in \C$, its complex conjugate is $\bar{\omega}$, and its modulus is $|\omega| = \sqrt{\omega \bar{\omega}}$. 
The complex conjugate of a vector $\bw \in \C^{m}$, denoted $\bar{\bw}$, is formed by element-wise conjugation. 
We will let $*$ be the Hermitian transpose.

A helpful matrix decomposition is the singular value decomposition (SVD), which decomposes any matrix $\bM \in \C^{m \times n}$ into factors $\bM = \bU \bSigma \bV^*$, where $\bU \in \C^{m \times m}$ and $\bV \in \C^{n \times n}$ are unitary. 
A matrix $\bQ \in \C^{m \times m}$ is unitary if $\bQ^* \bQ = \bQ \bQ^* = \bI$. 
The diagonal matrix $\bSigma \in \R^{m \times n}$ contains real, non-negative singular values $\sigma_1(\bM) \ge \sigma_2(\bM) \ge \cdots \ge \sigma_r(\bM) \ge 0$, where $r = \mbox{rank}(\bM) \le \min(m,n)$. 
The SVD can define a pseudo-inverse for rectangular matrices. 
In particular, a short-fat matrix $\bM \in \C^{m \times n}$ with $ m \le n$ has a right-inverse $\bM^\dagger = \bM^*(\bM \bM^*)^{-1}$, which is: $\bM^\dagger = \bV \bSigma^\dagger \bU^*$.

A special class of matrices are Hermitian matrices, matrices $\bA \in \C^{m \times m}$ such that $\bA = \bA^*$. 
The spectral theorem applies, and $\bA = \bP \bLambda \bP^{-1}$, where $\bLambda$ is a real-valued diagonal matrix composed of the eigenvalues of $\bA$, and where $\bP^{-1} = \bP^*$ is the unitary eigenvector matrix that diagonalizes $\bA$. 
An immediate consequence of any matrix $\bM$ having an SVD is that its singular values relate to the eigenvalues of $\bA = \bM \bM^*$. 
That is, $\bA = \bM \bM^* = \bU \bSigma^2 \bU^*$, meaning $\sigma_j(\bM) = \sqrt{\lambda_j(\bA)}$ for $j \in \{1,\cdots,m\}$.

To establish deterministic bounds on the matrix mappings of interest, a helpful metric is the spectral radius, which is equivalent to the operator norm and the largest singular value. 
That is, $\|\bM\|_{\op} = \sigma_{\max}(\bM)$. 
A notable consequence of this is that $\|\bM^\dagger\|_{\op} = \sigma_{\min}(\bM)$. 
With these values, we define the condition number, which measures the sensitivity of a matrix, as $\kappa(\bM) = \|\bM\|_{\op} \cdot \|\bM^\dagger\|_{\op} = \frac{\sigma_{\max}(\bM)}{\sigma_{\min}(\bM)}$.

Another metric for matrix sensitivity is the determinant. 
It can quantify the degree to which a square matrix, as a mapping, distorts the volume of the image of a unit ball. 
The determinant of a square matrix $\bA \in \C^{m \times m}$ with eigenvalues $\{\lambda_j\}_{j = 1}^m$ is the product of its $m$ eigenvalues; that is, $\det(\bA) = \prod_{j = 1}^m \lambda_j(\bA)$. 
While rectangular matrices have no formal determinant, since every matrix $\bM \in \C^{m \times n}$ has singular values $\{\sigma_j\}_{j = 1}^r$, which relate to the eigenvalues of $\bA =\bM \bM^*$ in that $\sigma_j(\bM) = \sqrt{\lambda_j(\bA)}$, we can define a generalized determinant for rectangular matrices: 
\begin{align}\label{eq:Generalized_Determinant}
    S(\bM) := \sqrt{|\det(\bA)|} = \prod_{j = 1}^m \sigma_j(\bM).
\end{align}
Naturally, just as $\det(\bA^{-1}) = \det(\bA)^{-1}$, so too do we have $S(\bM^\dagger) = S(\bM)^{-1}$. 
Thus, the generalized determinant can offer another measure of the sensitivity of a matrix.

For the condition number to be finite and the determinant to be well-defined, we need to know when $\sigma_{\min}(\bM) > 0$ and  $\sigma_{\max}(\bM) < + \infty$. 
Having such conditions would allow us to determine how stable a mapping is.
In many cases, explicitly computing singular values, especially for general classes of matrices, is impossible, so we resort to looser estimates in the form of bounds on the singular values. 
One way to bound the spectrum of a matrix is using Gerschgorin's circle theorem. 
We will use $j$ for rows and $k$ for columns to avoid using $i$ as an index and creating ambiguity when working with complex numbers.

\begin{theorem}\label{thm:Gerschgorin}
    (Gerschgorin ~\cite{Golub_1996}) Let $\bA \in \C^{m \times m}$ have entries $a_{jk}$ for $j, k \in \{1,...,m\}$. For each $j \in \{1, ..., m\}$, let $R_j$ be the sum of the absolute values of the non-diagonal entries in the $j$th row of $\bA$:
    \begin{align}
        R_j = \sum_{k = 1, k \neq j}^m |a_{jk}|
    \end{align}
    Let $D(a_{jj}, R_j) \subseteq \C$ be a closed disc of radius $R_j$ centered at $a_{jj}$. 
    Then, every eigenvalue of $\bA$ lies within at least one of the discs $D(a_{jj},R_j)$.
\end{theorem}

Gerschgorin's theorem defines sets guaranteed to contain the eigenvalues of $\bA$ using only information about the entries of $\bA$. 
For this result to indicate whether $\bA$ is non-singular, we need to show that zero is not in the spectrum of $\bA$, which a corollary of Theorem \ref{thm:Gerschgorin} can ensure:

\begin{corollary}\label{cor:Scalar_Diag_Dominance}
    A matrix $\bA \in \mathbb{C}^{m \times m}$ with entries $a_{jk}$ is \textit{strictly diagonally dominant} if 
    \begin{align}
        \sum_{k = 1, k \neq j}^m |a_{jk}| < |a_{jj}| \quad \mbox{for} \quad j \in \{1,..., m\}
    \end{align} 
    If $\bA$ is strictly diagonally dominant, then $0$ is not contained in the Gerschgorin discs, in which is contained the spectrum of $\bA$, and thus $\bA$ is non-singular.
\end{corollary}

Since we will also consider block matrices, having a sufficient condition of strict diagonal dominance for block matrices would be useful for determining when they are invertible. 
Theorem 3.1 of ~\cite{Echeverria_2018}, which generalizes the result of ~\cite{Feingold_1962} and ~\cite{Benzi_2017}, provides such a definition and a theorem about matrix non-singularity.

\begin{definition}\label{def:Block_Diag_Dominance}
    (Definition 2.1 ~\cite{Echeverria_2018}) We say a block matrix $\A = [\bA_{jk}]$ with blocks $\bA_{jk} \in \C^{m \times m}$ for $j,k \in \{1,...,n\}$ is called \textit{row block diagonally dominant} with respect to the chosen matrix norm $\|\cdot\|$ when the blocks $\bA_{jj}$ are non-singular, and
    \begin{align}\label{eq:Block_Diag_Dominance}
        \sum_{k = 1, k \neq j}^n \|\bA_{jj}^{-1} \bA_{jk}\| \le 1 \quad \mbox{for} \quad j \in \{1,..., n\}
    \end{align}
\end{definition}

\begin{lemma}\label{lemma:Block_Diag_Dominance_Invertible}
    (Lemma 3.1 ~\cite{Echeverria_2018}) If a matrix $\A$ is row block strictly diagonally dominant as in Definition \ref{def:Block_Diag_Dominance}, then $\A$ is non-singular.
\end{lemma}

\begin{theorem}\label{thm:Block_Gerschgorin}
    (Block Gerschgorin; Corollary 3.2 ~\cite{Echeverria_2018}) If $\A$ is as in Definition \ref{def:Block_Diag_Dominance}, and $\lambda$ is an eigenvalue of $\A$, then there exists at least one $j \in \{1,..., n\}$ with
    \begin{align}
        G_j = \{z \in \C: \sum_{k = 1, k \neq j}^n \|(\bA_{jj} - z I)^{-1} \bA_{jk}\| \ge 1\}
    \end{align}
    That is, the spectrum of $\A$ lies in the union of discs: $\sigma(\A) := \{\lambda_j(\A)\}_{j = 1}^{mn} \in \cup_{j = 1}^n G_j$.
\end{theorem}

In most cases, the norm for Theorem \ref{thm:Block_Gerschgorin} will be the operator norm $\|\cdot\|_{\op}$.

\section{Results}

The goal of this section is to determine the conditions under which the delay matrix $\M_{n,m}$ is full row-rank so that the LRNN can properly embed the dynamics of the input time series into the output latent space.
We build to the desired result by considering three cases of increasing complexity and generality. 
The first is the scalar case, for which we determine sufficient conditions for the non-singularity of the delay matrix $\M_{n,m}$ based on the weight parameter $\omega$, explicitly compute the singular values, bound the condition number, and calculate the generalized determinant. 
In the second case, where the weight matrix $\bW \in \C^{m \times m}$ is Hermitian, we show that all previous results generalize from the scalar case. 
In the third case, we consider arbitrary $\bW \in \C^{m \times m}$ and define sufficient conditions--albeit rather sub-optimal ones--for the non-singularity of $\M_{n,m}$, establish bounds on its singular values and thus its condition number in certain regimes. 
In each case, without loss of generality, we will ignore the negative sign on the weight matrix $\bW$.

\subsection{Scalar Case}

In the scalar case, we consider single-variable LRNNs of the form $h_{k+1} = \omega h_{k} + x_k + b$, where $h_k, x_k, \omega, b \in \C$. 
By collecting $n$ lagged observations and concatenating them, we form the delay matrix $\bM_n \in \C^{n \times (n+1)}$:
\begin{align}
    \bM_n: =
    \begin{bmatrix}
        1 & \omega & 0 & \cdots & 0 \\
        0 & 1 & \omega & \ddots & \vdots \\
        \vdots & \ddots & \ddots & \ddots & 0 \\
        0 & \cdots & 0 & 1 & \omega
    \end{bmatrix}  
    \in \C^{n \times n+1}
    \label{eq:M_Scalar}
\end{align}

To characterize $\bM_n$, we study the following Toeplitz tri-diagonal matrix:
\begin{align}
    \bA_n := \bM_n \bM_n^* = 
    \begin{bmatrix}
        1 + |\omega|^2 & \omega & 0 & \cdots & 0 \\
        \bar{\omega} & 1 + |\omega|^2 & \omega & \ddots & \vdots \\
        0 & \bar{\omega} & \ddots & \ddots & 0 \\
        \vdots & \ddots & \ddots & \ddots & \omega \\
        0 & \cdots & 0 & \bar{\omega} & 1+ |\omega|^2
    \end{bmatrix} \in \C^{n \times n}
    \label{eq:MM_Scalar}
\end{align}

\subsubsection{Singular Values and Condition Number}

The singular values of $\bM_n$ satisfy $\sigma_j(\bM_n) = \sqrt{\lambda_j(\bA_n)}$ for $j \in \{1, ..., n\}$. 
For Toeplitz tri-diagonal matrices
\begin{align}
    \bT_n = \begin{bmatrix}
        a & b &  &\\
        c & a & \ddots & \\
          & \ddots & \ddots & b \\
          &  & c & a &
    \end{bmatrix} \in \C^{n \times n} 
    \quad \mbox{for} \quad a,b,c \in \C,
\label{eq:Toeplitz_Tridiag_Mat}
\end{align}
it is known that its eigenvalues are ~\cite{Smith_1985, Kulkarni_1999}:
\begin{align}
    \lambda_j(\bT_n) = a + 2 \sqrt{bc} \cos\left(\frac{j \pi}{n + 1}\right) \quad \mbox{for} \quad j \in \{1, ..., n\}.
    \label{eq:Toeplitz_Tridiag_Mat_Eigs}
\end{align}

This allows us to find $\sigma_j(\bM_n)$.

\begin{proposition}         
    \label{prop:Scalar_Mat_Singular_Values}
    Let $\bM_n$ be as in \eqref{eq:M_Scalar}. Then,
    \begin{align}
        \sigma_j(\bM_n) = \sqrt{|\omega|^2 + 2 |\omega| \cos \left(\frac{j \pi}{n + 1}\right) + 1} \quad \mbox{for} \quad j \in \{1,..., n\}.
    \end{align}
\end{proposition}

\begin{proof}
    Apply \eqref{eq:Toeplitz_Tridiag_Mat_Eigs} to \eqref{eq:MM_Scalar} by letting $a = 1 + |\omega|^2, b = \omega$ and $c = \bar{\omega}$. 
\end{proof}

With the singular values of $\bM_n$, we can bound its condition number.

\begin{proposition}\label{prop:Scalar_Mat_Cond}
    Let $\bM_n$ be as in \eqref{eq:M_Scalar}. 
    Then,
    \begin{align}
        1 \le \kappa(\bM_n) \le \begin{cases}
            \bigg\vert\frac{|\omega| + 1}{|\omega| - 1} \bigg\vert \qquad \mbox{for} \quad |\omega| \in \{(0,1)\cup(1,\infty)\} \\
            \frac{2}{\pi}(n+1) \quad \mbox{for} \quad |\omega| = 1
        \end{cases}
    \end{align}
    For $|\omega| \in \{(0,1)\cup(1,\infty)\}$, $\sigma_{\max}(\bM_n) \le | |\omega| - 1|$ and $\sigma_{\min}(\bM_n) \ge | |\omega| - 1|$. 
\end{proposition}

\begin{proof}
    Using Theorem \ref{thm:Gerschgorin}, we can establish upper and lower bounds on $\sigma_{\max}(\bM_n)$ and $\sigma_{\min}(\bM_n)$, respectively, to compute $\kappa(\bM_n)$.
    To ensure $\kappa(\bM_n) < +\infty$, we require $\sigma_{\min}(\bM_n) > 0$, which is a sufficient condition for $\bA_n$ to be strictly row diagonally dominant  (and, since $\bA_n$ is Hermitian, this would also imply strict column diagonal dominance). 
    The magnitudes of the diagonal entries of $\bA_n$ are:
    \begin{align*}
        |a_{jj}| = |1 + |\omega|^2| = 1 + |\omega|^2 \quad \mbox{for} \quad j \in \{1, ..., n\}.
    \end{align*}
    The radii of the Gerschgorin disks are:
    \begin{align*}
        R_j = \begin{cases}
            |\omega| \qquad \mbox{for} \quad j \in \{1,n\} \\
            2 |\omega| \, \, \, \quad \mbox{for} \quad j \in \{2,..., n-1\}.
        \end{cases}
    \end{align*}
    When $j \in \{1,n\}$, $a_{jj} = 1 + |\omega|^2 > |\omega| = R_j$ for all $\omega \in \C$, so we only need to consider $j \in \{2,..., n-1\}$. 
    For strict diagonal dominance of $\bA_n$, we require $1 + |\omega|^2 > 2 |\omega|$, which holds for $\omega \in \{z \in \C | |z| \neq 1\}$--the entire complex plane excluding $\omega$ on the unit circle. 
    If we avoid $\omega \in \{z \in \C| |z| = 1\}$, then $\bA_n$ is strictly diagonally dominant, and
    \begin{align*}
        &|\lambda_j(\bA_n) - a_{ii}| \le R_j \quad \mbox{for} \quad j \in \{2,..., n-1\}.
    \end{align*}
    Since $\sigma_j(\bM_n) = \sqrt{\lambda_j(\bA_n)}$,
    \begin{align*}
        |\sigma_j(\bM_n)^2 - (1 + |\omega|^2)| \le 2 |\omega| \quad \mbox{for} \quad j \in \{2,..., n-1\}.
    \end{align*}
    This gives rise to the two-sided inequality, which also holds for $j = 1$ and $j = n$:
    \begin{align*}
        ||\omega| - 1| \le \sigma_j(\bM_n) \le ||\omega| + 1| \quad \mbox{for} \quad j \in \{1,..., n\}
    \end{align*}
    Thus, $\sigma_{\max}(\bM_n) \le ||\omega| + 1|$ and $\sigma_{\min}(\bM_n) \ge ||\omega| - 1|$. 
    In this regime,
    \begin{align*}
        \kappa(\bM_n) = \frac{\sigma_{\max}(\bM_n)}{\sigma_{\min}(\bM_n)} = \frac{||\omega| + 1|}{||\omega|-1|}.
    \end{align*}
    When $\omega = 1$, $\bA_n$ is not strictly diagonally dominant, yet it resembles the matrix for the discrete Laplacian ~\cite{LeVeque_2007}. 
    From \eqref{eq:Toeplitz_Tridiag_Mat_Eigs}, its eigenvalues are:
    \begin{align*}
        \lambda_j(\bA_n) = 2\left(1- \cos\left(\frac{j \pi}{n+1}\right)\right) = 4 \sin^2 \left(\frac{j \pi}{2(n+1)}\right) \quad \mbox{for} \quad j \in \{1, ..., n\},
    \end{align*}
    and $\sigma_j(\bM_n) = 2 \big \vert \sin\left(\frac{j \pi}{2(n+1)}\right) \big \vert$. 
    When $j = n$, $\sigma_{\max}(\bM_n) = 2 \big \vert \sin\left(\frac{n \pi}{2(n+1)}\right) \big \vert$, and when $j = 1$, $\sigma_{\min}(\bM_n) = 2 \big \vert \sin\left(\frac{\pi}{2(n+1)}\right) \big \vert$, so:
    \begin{align*}
        \kappa(\bM_n) 
        = \left| \frac{\sin\left(\frac{n \pi}{2(n+1)}\right)}{\sin\left(\frac{\pi}{2(n+1)}\right)} \right|
        = \left|\cot\left(\frac{\pi}{2(n+1)}\right) \right|.
    \end{align*}
    
    To bound $\kappa(\bM_n)$ independently of $n$, note that $\tan(x)$ has a convergent Taylor series about $x = 0$, valid for $|x| < \frac{\pi}{2}$, given by $\tan(x) = \sum_{k = 1}^\infty \frac{T_k}{(2k-1)!} x^{2k-1}$, where $T_k = \frac{4^k (4^k -1) |B_{2k}|}{2k}$, with $B_{k}$ being the $k$th Bernoulli number ~\cite{Weisstein}.
    Since each $T_k \ge 0$,
    \begin{align*}
        \bigg|\cot\left(\frac{1}{x}\right)\bigg| = \frac{1}{|\tan\left(\frac{1}{x}\right)|} = \frac{1}{|x + \cO(x^3)|} < \frac{1}{x}.
    \end{align*}
    
    Provided $|x| = \big|\frac{\pi}{2(n+1)}\big| < \frac{\pi}{2}$, the above inequality holds, which is true for any integer $n > 0$.
    Thus, since $\cot\left(\frac{\pi}{2(n+1)}\right)$ is bounded above by an oblique asymptote, 
    \begin{align*}
        \kappa(\bM_n) < \frac{2}{\pi} (n+1).
    \end{align*}
    
    Hence, when $\omega = 1$, $\kappa(\bM_n)$ grows linearly with the number of lags $n$.
\end{proof}

\begin{remark}
    What Proposition \ref{prop:Scalar_Mat_Cond} implies for LRNN is:
    \begin{itemize}
        \item When $|\omega| \neq 1$, the bound on $\kappa(\bM_n)$ is independent of the number of embedding lags $n$. 
        Increasing the number of lags, even indefinitely, will not lead the delay matrix to be catastrophically ill-conditioned, at least when $|\omega| \neq 1$. 
        \item When $\omega = 0$, $\bA_n$ reduces to the $n \times n$ identity matrix and $\kappa(\bM_n) = 1$. 
        The mapping $\bM_n$ is as stable as it can be, yet this would mean no information of the past lags $h_{k-1}$ would be used to determine the present value $h_k$ of the series.
        In the context of \eqref{eq:LRNN}, the explicit solution for the latent space states can be read from $\bM_n \Psi(\bh_k) = \Phi(\bx_k) + b \boldsymbol{1}^\top$, where $\bM_n$ is the $n \times n$ identity matrix with a column of zeros appended.
        The explicit values would be $\{h_{k-\ell}\}_{\ell = 0}^n = y_{k-(\ell+1)} + b$, where $h_{k - (n+1)}$ is a free-variable, which we would fix as $y_{k-(n+2)}$ so that, as the delay-coordinate vector shifts for different values of $k$, the flow will be consistent in the latent space.
        \item The regime where $|\omega| \ll 1$ corresponds to a small weighting parameter in the LRNN and implies past states receive less attention than the present one.
        \item When $\omega = 1$, $\kappa(\bM_n)$ grows linearly with the number of lags $n$. 
        Yet, $\omega = 1$ is not the only potential issue. 
        As $\omega \to 1^-$ and as $\omega \to 1^+$, $\kappa(\bM_n) \to \infty$, so $\bM_n$ becomes less stable.
        \item The regime where $|\omega| \gg 1$ corresponds to a large weighting parameter for the LRNN and would mean previous states receive greater weighting than the present state.
        Note that $\lim_{|\omega| \to \infty} \kappa(\bM_n) = 1$, meaning $\bM_n$ is relatively well-behaved, especially as the past states are weighted more heavily. 
    \end{itemize}
\end{remark}

\begin{figure}
    \centering
    \begin{subfigure}[t]{.5\textwidth}
        \centering
        \includegraphics[width=\linewidth]{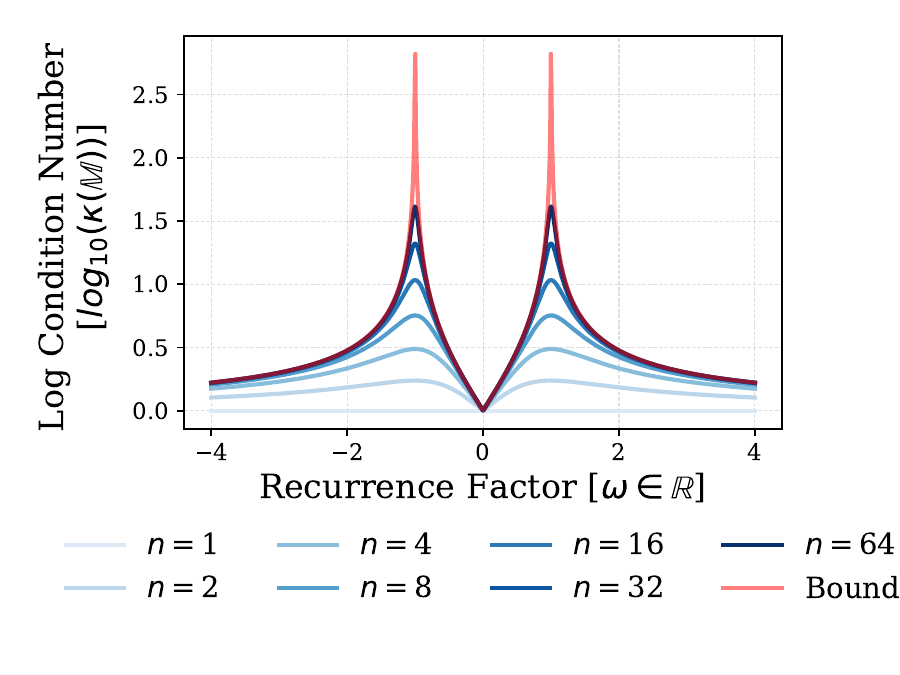}
        \vspace{-1.25cm}
        \caption{\centering}
        \label{fig:RecMat_Scalar_Cond_Gen}
    \end{subfigure}%
    \begin{subfigure}[t]{.5\textwidth}
        \centering
        \includegraphics[width=\linewidth]{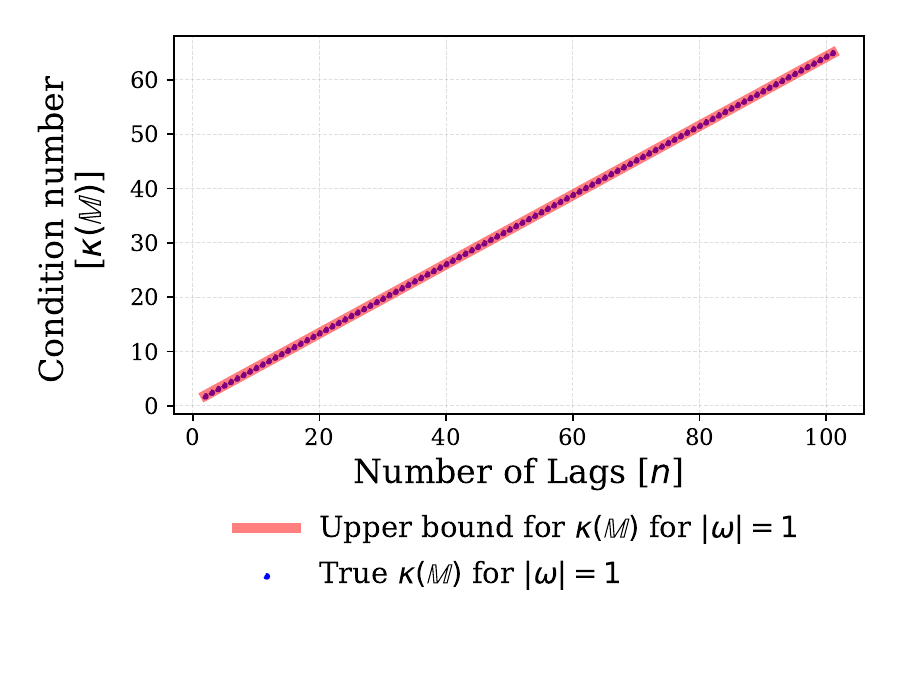}
        \vspace{-1.25cm}
        \caption{\centering}
        \label{fig:RecMat_Scalar_Cond_1}
    \end{subfigure}
    \caption{Log condition numbers $\kappa(\bM_n)$ for the scalar case of the delay matrix $\bM_n$ as in \eqref{eq:M_Scalar} with varying lags $n$, along with the bound as in Proposition \ref{prop:Scalar_Mat_Cond}, are plotted against the weight parameter $\omega$. 
    (a) For $\omega$ away from $|\omega| \approx 1$, $\kappa(\bM_n)$ is of relatively low magnitude, indicating $\bM_n$ is a stable embedding. For $|\omega|$ near $1$, $\bM_n$ becomes ill-conditioned. 
    (b) For $|\omega| = 1$, the bound on $\kappa(\bM_n)$ increases linearly on the number of lags $n$; increasing $n$ makes $\bM_n$ an unstable embedding.}
    \label{fig:RecMat_Scalar_Conds}
\end{figure}

\subsubsection{Determinant}

We now compute $S(\bM_n) = \sqrt{|\det(\bA_n)|}$ as a heuristic for the sensitivity of $\bM_n$, in that it can quantify the degree to which $\bM_n$ distorts the output space relative to the input space. 
The determinant of a Toeplitz tri-diagonal matrix $\bT_n$ \eqref{eq:Toeplitz_Tridiag_Mat} from ~\cite{Cinkir_2014} is:
\begin{align}\label{eq:Tridiag_Toeplitz_Det}
    \det(\bT_n) = \frac{1}{\sqrt{a^2-4bc}} \left(\left(\frac{a + \sqrt{a^2 - 4bc}}{2}\right)^{n+1} - \left(\frac{a - \sqrt{a^2 - 4bc}}{2}\right)^{n+1} \right)
\end{align}
for $a^2 - 4bc \neq 0$, and when $a^2 - 4bc = 0$, the determinant is:
\begin{align}\label{eq:Tridiag_Toeplitz_Det_0}
    \det(\bT_n) = (n+1) \left(\frac{a}{2}\right)^n.
\end{align}
Applying \eqref{eq:Tridiag_Toeplitz_Det} and \eqref{eq:Tridiag_Toeplitz_Det_0} to \eqref{eq:MM_Scalar} provides the following result.

\begin{proposition} \label{prop:Scalar_Mat_Det_Bds}
    Let $\bM_n$ be as in \eqref{eq:M_Scalar}. Then, 
    \begin{align}
        S(\bM_n) = \sqrt{|\det(\bA_n)|} \le \begin{cases}
            \frac{1}{\sqrt{1 - |\omega|^2}} \qquad |\omega| < 1 \\
            \sqrt{n+1} \qquad |\omega| = 1 \\
            \frac{|\omega|^n}{\sqrt{1 - \frac{1}{|\omega|^2}}} \qquad |\omega| > 1
        \end{cases}
    \end{align}
\end{proposition}

\begin{proof}
    For $\bA_n$ from \eqref{eq:MM_Scalar}, let $a = 1 + |\omega|^2$, $b = \omega$ and $c = \bar{\omega}$. 
    The discriminant is: $a^2 - 4bc = (1+|\omega|^2)^2 - 4 \omega \bar{\omega} = (1 - |\omega|^2)^2$. 
    When $|\omega| \neq 1$, $a^2 - 4bc \neq 0$, so
    \begin{align*}
        \det(\bA_n) &= \frac{1}{1 - |\omega|^2} \left(\left(\frac{1 + |\omega|^2 + (1 - |\omega|^2)}{2}\right)^{n+1} - \left(\frac{1 + |\omega|^2 - (1 - |\omega|^2)}{2}\right)^{n+1} \right) \\
        &= \frac{1 - |\omega|^{2n+2}}{1 - |\omega|^2} 
        = \sum_{k = 0}^n |\omega|^{2k},
    \end{align*}
    which is $S(\bM_n)^2$. 
    The following bounds come from the second-to-last expression in the above: for $|\omega| < 1$, $S(\bM_n) \le \frac{1}{\sqrt{1-|\omega|^2}}$, and for $|\omega| > 1$, $S(\bM_n) \le \frac{|\omega|^n}{\sqrt{1 - \frac{1}{|\omega|^2}}}$, since $\frac{1}{|\omega|} < 1$. 
    When $|\omega| = 1$, $\det(\bA_n) = (n+1) \left(\frac{1 + |\omega|^2}{2}\right)^n = (n+1)$.
\end{proof}

\begin{remark}
Proposition \ref{prop:Scalar_Mat_Det_Bds} leads to the following conclusions:
\begin{itemize}
    \item When $|\omega| < 1$, the bound on the generalized determinant $S(\bM_n)$ is independent of the number of lags $n$. 
    However, the bound itself is not necessarily bounded, since $S(\bM_n) \to \infty$ as $|\omega| \to 1^-$.
    \item When $|\omega| = 1$, the bound on $S(\bM_n)$ depends on $n$, but only linearly.
    \item When $|\omega| > 1$, the bound on $S(\bM_n)$ depends exponentially on $n$ and is effectively unbounded. 
    In this regime, $\bM_n$ severely distorts the space and is highly sensitive, despite the condition number being small in this regime. \\
\end{itemize}
Thus, for $\bM_n$ to be a stable embedding, it is most reasonable to consider $|\omega| \le 1$. 
Note that in \eqref{eq:LRNN_Delay_Coordinates_Sol} we consider $\bM_n^\dagger$, so it would make sense to consider $S(\bM_n)^{-1}$. 
The behavior is essentially the same, considering that $\bM_n$ being unstable in the sense that $S(\bM_n) \to \infty$ corresponds to $S(\bM_n^\dagger) \to 0$.
\end{remark}

\begin{figure}
    \centering
    \begin{subfigure}[t]{.5\textwidth}
        \centering
        \includegraphics[width=\linewidth]{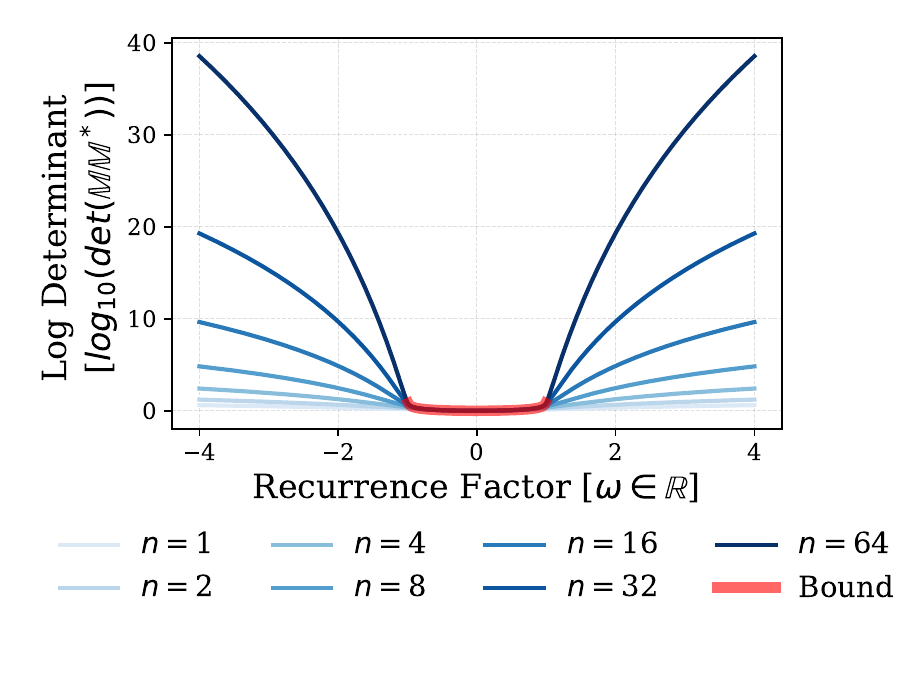}
        \vspace{-1.25cm}
        \caption{\centering}
        \label{fig:RecMat_Scalar_Dets}
    \end{subfigure}%
    \begin{subfigure}[t]{.5\textwidth}
        \centering
        \includegraphics[width=\linewidth]{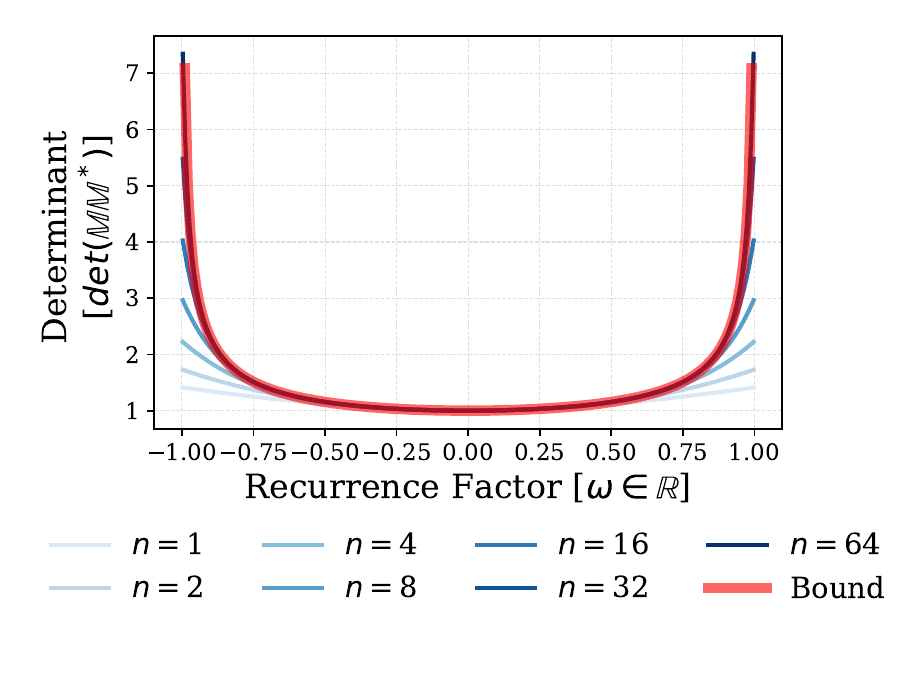}
        \vspace{-1.25cm}
        \caption{\centering}
        \label{fig:RecMat_Scalar_Dets_w1}
    \end{subfigure}
    \caption{Log generalized determinants $S(\bM_n)$ for the scalar case of delay matrices $\bM_n$ as in \eqref{eq:M_Scalar} with varying lags $n$, along with the bound for $|\omega| < 1$ as in Proposition \ref{prop:Scalar_Mat_Det_Bds}, are plotted against the weight parameter $\omega$. 
    (a) For $\omega \notin \{\omega : |\omega| < 1\}$, $S(\bM_n)$ becomes unbounded in an exponential fashion, so $|\omega| > 1$ produces an unstable delay matrix $\bM_n$. 
    (b) While for most $\omega \in (-1,1)$, $S(\bM_n)$ is near $1$, indicating $\bM_n$ preserves ``volume" and is thus a relatively stable embedding, as $|\omega| \to 1^-$, $S(\bM_n) \to \infty$ and the embedding destabilizes.}
    \label{fig:RecMat_Scalar_Dets_Panel}
\end{figure}

\subsection{Hermitian Block Case}

Having examined the scalar case, we extend the problem to the case where the LRNN is defined by a matrix $\bW$, allowing for more expressive interaction between various states. 
A case that lends itself well to analysis is when $\bW$ in \eqref{eq:LRNN} is Hermitian, as it is unitarily diagonalizable and $\bW = \bW^*$. 
Consider the block matrix $\M_{n,m}\in \C^{mn \times m(n+1)}$ with the $m \times m$ identity matrix $\bI$ repeated $n$ times along the main diagonal and Hermitian $\bW \in \C^{m \times m}$ along the super-diagonal, where without loss of generality we disregard the sign of $\bW$:
\begin{align}
    \M_{n,m}: =
    \begin{bmatrix}
        \bI & \bW & \bzero & \cdots & \bzero \\
        \bzero & \bI & \bW & \ddots & \vdots \\
        \vdots & \ddots & \ddots & \ddots & \bzero \\
        \bzero & \cdots & \bzero & \bI & \bW
    \end{bmatrix} 
    \in \C^{m n \times m (n+1)}.
    \label{eq:M_hermitian}
\end{align}
Like the scalar case, we wish to know the singular values, condition number, and determinant of $\M_{n,m}$, among other properties. 
As before, to more easily compute the singular values of $\M_{n,m}$, we consider the eigenvalues of the following matrix:
\begin{align}
    \A_{n,m} := \M_{n,m}\M_{n,m}^* = \begin{bmatrix}
        \bI + \bW^2 & \bW & \bzero & \cdots & \bzero \\
        \bW & \bI + \bW^2 & \bW & \ddots & \vdots \\
        \bzero & \bW & \ddots & \ddots & \bzero \\
        \vdots & \ddots & \ddots & \ddots & \bW \\
        \bzero & \cdots & \bzero & \bW & \bI + \bW^2
    \end{bmatrix} \in \C^{m n \times m n}.
\label{eq:MM_hermitian}
\end{align}

Since $\bW$ is Hermitian, it can be diagonalized as $\bW = \bU \bLambda \bU^*$, where $\bU \in \C^{m \times m}$ is unitary and $\bLambda$ is a diagonal matrix with the eigenvalues of $\bW$. 
Hence,
\begin{align}\label{eq:MM_Hermitian_Decomp}
\begin{split}
    \A_{n,m} = \U_{n,m}
    \begin{bmatrix}
        \bI + \bLambda^2 & \bLambda & \bzero & \cdots & \bzero \\
        \bLambda & \bI + \bLambda^2 & \bLambda & \ddots & \vdots \\
        \bzero & \bLambda & \ddots & \ddots & \bzero \\
        \vdots & \ddots & \ddots & \ddots & \bLambda \\
        \bzero & \cdots & \bzero & \bLambda & \bI + \bLambda^2
    \end{bmatrix}
    \U_{n,m}^*.
\end{split}
\end{align}
where $\U_{n,m} = \diag(\bU, \cdots, \bU)$, and $\T_{n,m}$ is the central tri-diagonal block matrix so that $\A_{n,m} = \U_{n,m} \T_{n,m} \U_{n,m}^*$. 
Since $\U_{n,m}$ is a block diagonal matrix composed of unitary matrices, it is also unitary.
Hence, $\T_{n,m}$ is similar to $\A_{n,m}$ and shares the same eigenvalues, so it suffices to examine $\T_{n,m}$ to determine the properties of $\A_{n,m}$.

We can further simplify $\T_{n,m}$, which is a block tri-diagonal matrix, the non-zero blocks of which are diagonal matrices, by transforming it into a block diagonal matrix with Toeplitz tri-diagonal blocks. 
To do so, we use a lemma adapted from ~\cite{Golub_2005}.

\begin{lemma}\label{lemma:Block_Mat_Rearrangement}
    Let $\T_{n,m}$ be a block tri-diagonal matrix:
    \begin{align*}
    \T_{n,m} = \begin{bmatrix}
        \bA_1 & \bB_1 & \cdots & \bzero \\
        \bC_1 & \bA_2 & \ddots & \vdots \\
        \vdots & \ddots & \ddots & \bB_{n-1} \\
        \bzero & \cdots & \bC_{n-1} & \bA_n
    \end{bmatrix} \quad 
    \begin{cases}
        \bA_j = \diag(a_{j1},...,a_{jm}) \quad j \in \{1,...,n\} \\
        \bB_j = \diag(b_{j1},...,b_{jm}) \quad j \in \{1,...,n-1\} \\
        \bC_j = \diag(c_{j1},...,c_{jm}) \quad j \in \{1,...,n-1\} \\
    \end{cases}
    \end{align*}
    Let $\D_{m,n} = \diag(\bT_1, \bT_2, \cdots, \bT_m)$ be a block diagonal matrix the tri-diagonal blocks of which are:
    \begin{align*}
    \bT_k = \begin{bmatrix}
        a_{1,k} & b_{1,k} & 0 & \cdots & 0 \\
        c_{1,k} & a_{2,k} & \ddots & \ddots & \vdots \\
        0 & \ddots & \ddots & \ddots & 0 \\
        \vdots & \ddots & \ddots & \ddots & b_{n-1,k} \\
        0 & \cdots & 0 & c_{n-1,k} & a_{n,k}
    \end{bmatrix}
    \quad \mbox{for } k \in \{1,..., m\}
    \end{align*}
    Then $\D_{m,n} = \mathbb{P}_{m,n} \T_{n,m} \mathbb{P}_{m,n}^\top$ for some unitary permutation $\mathbb{P}_{m,n}$, where the subscript denotes a re-framing of $n$ blocks from size $m$ to $m$ blocks of size $n$.
\end{lemma}

\begin{proof}
    To transform $\T_{n,m}$ into $\D_{m,n}$, we want to group the first elements of each of the $n$ size $m \times m$ blocks into an $n \times n$ tri-diagonal matrix, the second elements of each set of matrices as another $n \times n$ block, and so on. 
    Since we will permute rows and columns in the same way, the permutation matrix $\mathbb{P}_{m,n}$ we construct will diagonalize $\T_{n,m}$.
    By construction, the $\mathbb{P}_{m,n}$ that accomplishes this has entries:
    \begin{align}
         p_{jk} = 
         \begin{cases}
             1 \qquad \mbox{for} \quad k = (j-1)m + 1, j \in \{1,..., n\} \\
             0 \qquad \mbox{for} \quad k \neq (j-1)m + 1, j \in \{1,..., n\} \\
         \end{cases} 
    \end{align}
    Thus, $\D_{m,n} = \mathbb{P}_{m,n} \T_{n,m} \mathbb{P}_{m,n}^\top$.
\end{proof}

Applying Lemma \ref{lemma:Block_Mat_Rearrangement} to $\T_{n,m}$ from \eqref{eq:MM_Hermitian_Decomp} shows that 
\begin{align}
    \A_{n,m} = \U_{n,m} \mathbb{P}_{m,n} \D_{m,n} \mathbb{P}_{m,n}^* \U_{n,m}^*,
\end{align} 
with $\D_{m,n} = \diag(\bT_1, \bT_2, \cdots, \bT_m)$ and each tri-diagonal Toeplitz block $\bT_k$ being:
\begin{align} \label{eq:T_rearranged}
    \bT_k = \begin{bmatrix}
        1 + \lambda_k^2 & \lambda_k & 0 & \cdots & 0 \\
        \lambda_k & 1 + \lambda_k^2 & \ddots & \ddots & \vdots \\
        0 & \ddots & \ddots & \ddots & 0 \\
        \vdots & \ddots & \ddots & \ddots & \lambda_k \\
        0 & \cdots & 0 & \lambda_k & 1 + \lambda_k^2
    \end{bmatrix}
    \quad \mbox{for } k \in \{1,..., m\}
\end{align}

\begin{remark}\label{rmk:Herm_Mat_Factor}
    We can fully diagonalize $\A_{n,m}$, since each tri-diagonal sub-block, as a Hermitian, Toeplitz, tri-diagonal matrix is diagonalizable. 
    Let $\V_{m,n}$ be the diagonal block matrix that diagonalizes $\D_{m,n}$.
    Then,
    \begin{align}
        \A_{n,m} = \U_{n,m} \mathbb{P}_{m,n} \V_{m,n} \bLambda_{nm} \V_{m,n}^* \mathbb{P}_{m,n}^* \U_{n,m}^*
    \end{align}
    where $\bLambda_{nm}$ is now a diagonal matrix, rather than a block diagonal matrix, the entries of which are the eigenvalues of $\A_{n,m}$. 
    Consequently, 
    \begin{align}
        \A_{n,m}^{-1} = \U_{n,m} \mathbb{P}_{m,n} \V_{m,n} \bLambda_{nm}^{-1} \V_{m,n}^* \mathbb{P}_{m,n}^* \U_{n,m}^*,
    \end{align} 
    which enables a relatively fast computation of $\A_{n,m}^{-1}$ and of $\M_{n,m}^\dagger = \M_{n,m}^*\A_{n,m}^{-1}$.
\end{remark}

\subsubsection{Singular Values and Condition Number}

Since the singular values and therefore the condition number of a matrix remain invariant under unitary transformations, we can analyze the block diagonal matrix $\D_{m,n}$ associated with \eqref{eq:T_rearranged}.

\begin{proposition}\label{prop:Herm_Mat_Singular_Values}
    If $\bW$ has eigenvalues $\{\lambda_k\}_{k = 1}^m$, the eigenvalues of $\A_{n,m}$ are:
    \begin{align*}
        \lambda_{jk}(\A_{n,m}) = \lambda_k^2 + 2 \lambda_k \cos\left(\frac{j \pi}{n+1}\right) +  1 \quad \mbox{ for } j \in \{1,..., n\} \mbox{ and } k \in \{1, ... m\}.
    \end{align*}
\end{proposition}

\begin{proof}
    The eigenvalues of $\A_{n,m}$ match those of $\D_{m,n}$, which is the set of all eigenvalues of each block $\bT_k$, $k \in\{1,...,m\}$. 
    Since each $\bT_k$ is a Toeplitz, tri-diagonal matrix, applying \eqref{eq:Toeplitz_Tridiag_Mat_Eigs} with $a = 1 + \lambda_k^2$, $b = \lambda_k$ and $c = \lambda_k$ gives the result.
\end{proof}

Having found $\lambda(\A_{n,m})$ explicitly, we next bound the condition number of $\M_{n,m}$.

\begin{theorem} \label{thm:Block_Mat_Herm_Cond}
    Let $\M_{n,m}$ be as in \eqref{eq:M_hermitian}. 
    Then,
    \begin{align}
        \kappa(\M_{n,m}) \le \begin{cases}
            \frac{\sigma_{\max}(\bW) + 1}{\min_{j \in \{1,...,m\}} |\sigma_j(\bW) - 1|} \quad \mbox{ for } \quad \sigma_j(\bW) \neq 1 \mbox{ for } j \in  \{1,...,m\} \\
            \frac{(n+1)}{\pi} (\sigma_{\max}(\bW) + 1) \quad \mbox{ for } \quad \sigma_j(\bW) = 1 \mbox{ for some } j \in \{1,...,m\}
        \end{cases}
    \end{align}
\end{theorem}

\begin{proof}
    Denote by $k \in \{1, ..., n\}$ the index of the blocks of $\A_{n,m}$, and by extension $\T$, and let $j \in \{1,...,m\}$ serve as the index for the row number within each block. 
    The centers and radii of the Gerschgorin discs for $\T$ are:
    \begin{align*}
        & a_{k,jj} = 1 + \lambda_j(\bW)^2 \quad \, \, \, \mbox{for} \quad j \in \{1,..., n\} \\
        & R_{k,j} = 
        \begin{cases}
            |\lambda_j(\bW)| \, \, \, \qquad \mbox{for} \quad j \in \{1,n\} \\
            2|\lambda_j(\bW)| \qquad \mbox{for} \quad j \in \{2,..., n-1\}
        \end{cases}
    \end{align*}
    Since the contents of the blocks in each row are the same, the centers and radii are independent of the block number $k$, except when $k \in \{1,n\}$.
    When seeking upper and lower bounds on the eigenvalues, the cases when $k \in \{1,n\}$ are subsumed by the other cases.
    Thus, the upper bound on each singular value $j \in \{1,...,m\}$ is:
    \begin{align*}
        \sigma_j(\M_{n,m})^2 &\le 1 + \lambda_j(\bW)^2 + 2 |\lambda_j(\bW)| = (|\lambda_j(\bW)| + 1)^2 \\
        &\quad \implies \sigma_j(\M_{n,m}) \le |\lambda_j(\bW)| + 1
    \end{align*}
    Since $\bW$ is Hermitian, $|\lambda_j(\bW)| = \sqrt{\lambda_j(\bW \bW^*)} = \sigma_j(\bW)$. 
    Consequently,
    \begin{align*}
        \sigma_{\max}(\M_{n,m}) \le \sigma_{\max}(\bW) + 1.
    \end{align*}
    Likewise, Theorem \ref{thm:Gerschgorin} gives a lower bound on the eigenvalues of each block of $\M_{n,m}$:
    \begin{align*}
        \sigma_j(\M_{n,m})^2 & \ge 1 + \lambda_j(\bW)^2 - 2 |\lambda_j(\bW)| = (|\lambda_j(\bW)| - 1)^2 \\
        & \quad \implies \sigma_j(\M_{n,m}) \le ||\lambda_j(\bW)| - 1|
    \end{align*}
    and the minimal singular value of $\M_{n,m}$ is bounded below by:
    \begin{align*}
        \sigma_{\min}(\M_{n,m}) \ge \min_{j \in \{1,...,m\}}||\sigma_j(\bW)| - 1|.
    \end{align*}
    This gives a bound on the condition number, provided $\sigma_j(\bW) \neq 1$ for $j \in \{1,...,m\}$:
    \begin{align*}
        \kappa(\M_{n,m}) \le \frac{\sigma_{\max}(\bW) + 1}{\min_{j \in \{1,...,m\}} |\sigma_j(\bW) - 1|}.
    \end{align*}
    When $\sigma_j(\bW) = 1$ for at least one $j \in \{1,...,m\}$, then $\sigma_{\min}(\M_{n,m}) = 2 \left|\sin\left(\frac{\pi}{2(n+1)}\right) \right|$ and thus $\frac{1}{\sigma_{\min}(\M_{n,m})} < \frac{(n+1)}{\pi}$, and 
    \begin{align*}
        \kappa(\M_{n,m}) \le \frac{(n+1)}{\pi} (\sigma_{\max}(\bW) + 1).
    \end{align*}
\end{proof}

The above result suggests that, when $\bW$ is Hermitian, the modes can decouple when analyzing $\M_{n,m}$, so changing one eigenvalue of $\bW$ does not affect the other sets eigenvalues of $\M_{n,m}$. 
Thus, delay matrices $\M_{n,m}$ with Hermitian $\bW$ are robust to perturbation of the eigenvalues of $\bW$. 
Theorem \ref{thm:Block_Mat_Herm_Cond}, similar to the scalar case, shows the conditioning of $\M_{n,m}$ is independent of $n$ but also of $m$, the weight matrix size. 
It also shows that $\M_{n,m}$ becomes highly unstable when $\bW$ has eigenvalues close to magnitude $1$, as the condition number comes to depend on the lags $n$. 
This phenomenon of avoiding eigenvalues of magnitude $1$ may relate to avoiding aliasing.

\begin{remark}
Theorem \ref{thm:Block_Mat_Herm_Cond} gives rise to the following insights:
    \begin{itemize}
        \item $\sigma_j(\bW) = 0$ is a permissible eigenvalue and will almost never affect the condition number $\kappa(\M_{n,m})$. 
        Thus, when using a linear RNN, adding additional dimensions to $\bh$ and thus to $\bW$ will not impact the conditioning of the embedding, as the additional states, if left ``un-used" will have $\sigma_j(\bW) = 0$.
        \item The condition number bound is effectively independent of the dimension of the latent state space $m$ as it is chiefly dependent on the maximal singular value of $\bW$ and the singular value of $\bW$ closest to $1$.
        \item $\sigma_{\max}(\bM_n), \sigma_{\min}(\bM_n)$ and $\kappa(\bM_n)$, in the scalar case of Proposition \ref{prop:Scalar_Mat_Cond} are recoverable from Theorem \ref{thm:Block_Mat_Herm_Cond} with $m = 1$. 
    \end{itemize}
\end{remark}

\begin{figure}[h!]
    \centering
    \begin{subfigure}[b]{0.45\textwidth}
        \centering
        \includegraphics[width=\textwidth]{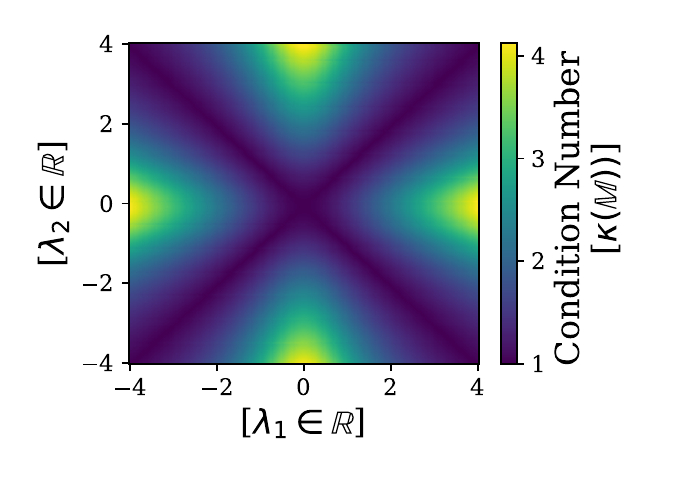}
        \vspace{-9mm}
        \caption{\centering $n = 2$}
        \label{fig:HermMat_2D_Cond_2}
    \end{subfigure}    
    \hfill
    \begin{subfigure}[b]{0.45\textwidth}
        \centering
        \includegraphics[width=\textwidth]{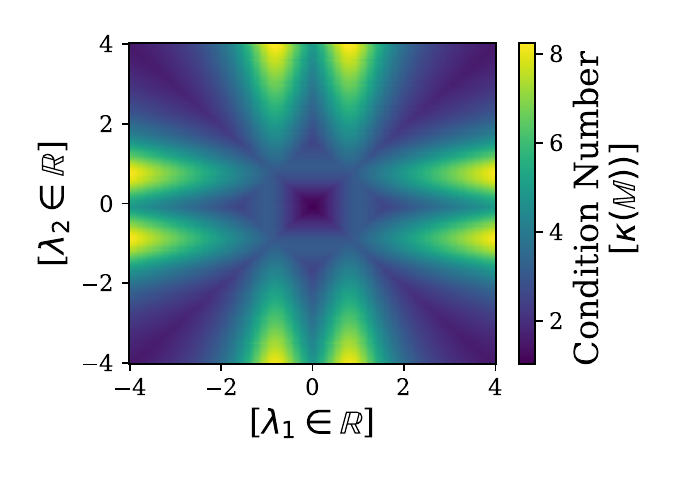}
        \vspace{-9mm}
        \caption{\centering $n = 8$}
        \label{fig:HermMat_2D_Cond_8}
    \end{subfigure}
    \vspace{0.1cm}
    \begin{subfigure}[b]{0.45\textwidth}
        \centering
        \includegraphics[width=\textwidth]{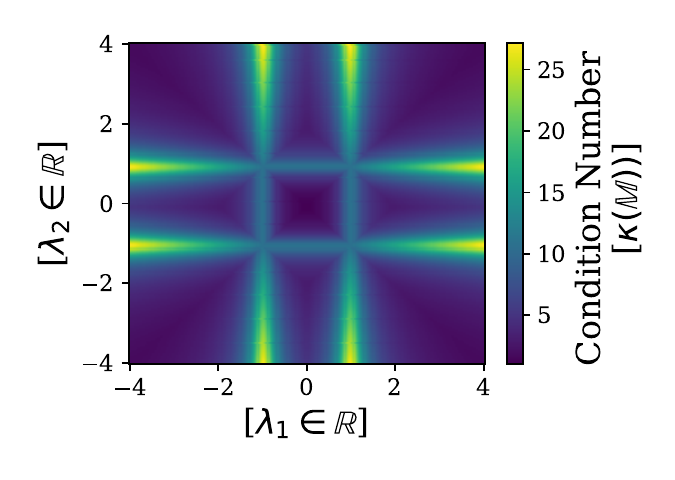}
        \vspace{-9mm}
        \caption{\centering $n = 32$}
        \label{fig:HermMat_2D_Cond_32}
    \end{subfigure}
    \hfill
    \begin{subfigure}[b]{0.45\textwidth}
        \centering
        \includegraphics[width=\textwidth]{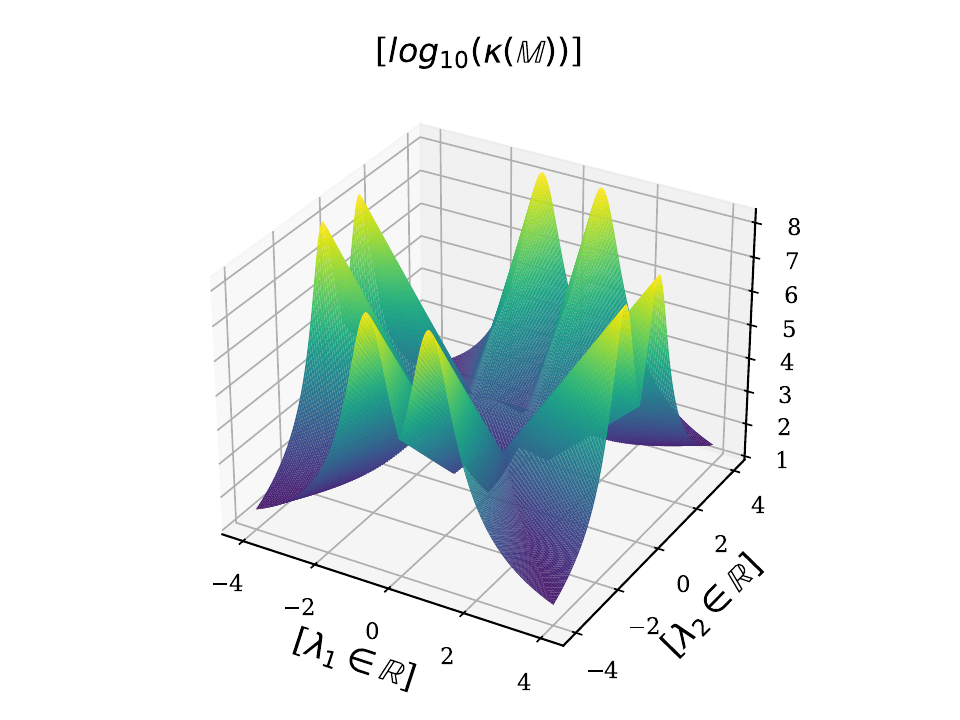}
        \vspace{-9mm}
        \caption{\centering $n = 4$}
        \label{fig:HermMat_2D_Cond_3D}
    \end{subfigure}
    \caption{The log of the condition number of $\M_{n,m}$, $\kappa(\M_{n,m})$, for Hermitian $\bW \in \C^{2 \times 2}$ is shown for varying choices of $\lambda_1(\bW)$ and $\lambda_2(\bW)$ in cases where the number of lags is (a) $n = 2$, (b) $n = 8$, and (c) $n = 32$.
    $\M_{n,m}$ is near-singular when $|\lambda_1(\bW)|$ or $|\lambda_2(\bW)|$ are close to $1$, and its ill-conditioning increases as $n$ increases.
    Notably, $\kappa(\M_{n,m})$ is small within the unit square and away from eigenvalues of magnitude $1$.
    (d) Noting that the structure of the plot in each direction closely resembles the scalar case \ref{fig:RecMat_Scalar_Conds}, it follows that the $m$-dimensional case for $\bW$ is the graph crossed with itself $m$ times.
    }
    \label{fig:HermMat_2D_Cond}
\end{figure}

\subsubsection{Determinant}

For another indicator of the degree to which $\M_{n,m}$ with Hermitian $\bW$ distorts the image of the unit ball, we find the determinant of $\A_{n,m}$.

\begin{theorem} \label{thm:Block_Mat_Herm_Det}
    Let $\M_{n,m}$ be as in \eqref{eq:M_hermitian}. 
    Then, letting $r$ denote the number of $\sigma_j(\bW)$ such that $\sigma_j(\bW) = 1$:
    \begin{align}
        S(\M_{n,m}):= \sqrt{|\det(\A_{n,m})|} & = (n+1)^{r/2} \cdot\prod_{j = 1}^{m-r} \left(\sum_{k = 0}^{n} \sigma_j(\bW)^{2k}\right)^{1/2}
    \end{align}
    and $S(\M_{n,m}) \ge 1$ for any set of weight matrix singular values $\{\sigma_j(\bW)\}_{j = 1}^{m}$.
\end{theorem}

\begin{proof}
    By Lemma \ref{lemma:Block_Mat_Rearrangement}, $\A_{n,m} = \U_{n,m} \mathbb{P}_{m,n} \D_{m,n} \mathbb{P}_{m,n}^* \U_{n,m}^*$, where $\D_{m,n}$ is the block diagonal matrix the blocks of which are tri-diagonal matrices $\bT_k$ as in \eqref{eq:T_rearranged}, and $\U_{n,m}$ and $\mathbb{P}$ are unitary.
    Because the determinant of a product of matrices equals the product of determinants, and the determinant of unitary matrices such as $\U_{n,m}$ and $\mathbb{P}_{m,n}$ is one, and the determinant of a block diagonal matrix is the product of the determinants of its diagonal blocks,
    \begin{align}
        \det(\A_{n,m}) &= \det(\U_{n,m} \mathbb{P}_{m,n} \D_{m,n} \mathbb{P}_{m,n} \U_{n,m}^*) = \det(\D_{m,n}) = \prod_{j = 1}^m \det(\bT_j)
    \end{align}
    From the proof of the scalar case in Proposition \ref{prop:Scalar_Mat_Det_Bds}:
    \begin{align*}
        \det(\bT_j) = \sum_{k = 0}^m \sigma_j(\bW)^{2k}.
    \end{align*}
    Letting $r$ denote the number of eigenvalues of the form $\lambda \in \{\lambda | |\lambda(\bW)| = 1 \}$,
    \begin{align*}
        \det(\A_{n,m}) = (n+1)^r \cdot \prod_{j = 1}^{m-r} \left(\sum_{k = 0}^n |\lambda_j(\bW)|^{2k}\right).
    \end{align*}
\end{proof}

Having explicitly computed the determinant for the Hermitian case of the delay matrix, we can now establish bounds on it.

\begin{proposition} \label{prop:Block_Mat_Herm_Det_Bds}
    Let $\M_{n,m}$ be as in \eqref{eq:M_hermitian}. 
    For $\sigma_{\max}(\bW) \le 1$, an upper bound on the determinant is:
    \begin{align}
        S(\M_{n,m}) = \sqrt{|\det(\A_{n,m})|} \le 
        \begin{cases}
            \left(\frac{1}{\sqrt{\min_{j \in \{1,..., m\}}(1 - \sigma_j(\bW)^2)}}\right)^{m} \quad \sigma_{\max}(\bW) < 1 \\
            \left(\sqrt{n+1}\right)^{m} \qquad \qquad \qquad \qquad \sigma_{\max}(\bW) = 1 \\
            n^{m/2} \sigma_{\max}(\bW)^{nm} \qquad \qquad \qquad \sigma_{\max}(\bW) > 1
        \end{cases}
    \end{align}
\end{proposition}

\begin{proof}
    Recall that the Arithmetic Mean - Geometric Mean (AM-GM) inequality states that for any list of $m$ nonnegative real numbers $\{x_1,..., x_m\}$, it follows that:
    \begin{align*}
        \left(\prod_{j = 1}^m x_j\right)^{1/m} \le \frac{1}{m} \sum_{j = 1}^m x_i,
    \end{align*}
    with equality when $x_1 = x_2 = \cdots = x_m$. 
    Since $|\det(\A_{n,m})|$ is a product of positive real numbers $x_j = \sum_{k = 0}^{n} \sigma_j(\bW)^{2k}$ for $j \in \{1,..,m\}$, by the AM-GM inequality:
    \begin{align*}
        |\det(\A_{n,m})| \le \left(\frac{1}{m} \sum_{j = 1}^m \sum_{k = 0}^{n} \sigma_j(\bW)^{2k} \right)^m = \left(\sum_{k = 0}^n \|\bW^k\|_F^2 \right)^m
    \end{align*}
    If $\sigma_{\max}(\bW) = 1$: 
    \begin{align*}
        |\det(\A_{n,m})| \le (n+1)^m
    \end{align*}
    If $\sigma_{\max}(\bW) < 1$:
    \begin{align*}
        |\det(\A_{n,m})| &\le \left(\frac{1}{m} \sum_{j = 1}^m \frac{1 - \sigma_j(\bW)^{2(n+1)}}{1 - \sigma_j(\bW)^2} \right)^m \le \left(\frac{1}{m} \sum_{j = 1}^m \frac{1}{1 - \sigma_j(\bW)^2} \right)^m \\
        &\le \left(\frac{1}{\min_{j \in \{1,.., m\}}(1 - \sigma_j(\bW)^2)}\right)^m
    \end{align*}
    If $\sigma_{\max}(\bW) > 1$, then:
    \begin{align*}
        |\det(\A_{n,m})| \le \left(n \sigma_{\max}(\bW)^{2n}\right)^m.
    \end{align*}
\end{proof}

\begin{remark}
    When $m = 1$, the exact determinant result for Hermitian $\bW$ in Theorem \ref{thm:Block_Mat_Herm_Det} and the determinant bounds in Proposition \ref{prop:Block_Mat_Herm_Det_Bds} agree with the scalar case as in Proposition \ref{prop:Scalar_Mat_Det_Bds}. 
    For $m > 1$, the sensitivity compounds according to $m$, the size of the weight matrix $\bW$.
\end{remark}

\begin{figure}[h!]
    \centering
    \begin{subfigure}[b]{0.48\textwidth}
        \centering
        \includegraphics[width=\textwidth]{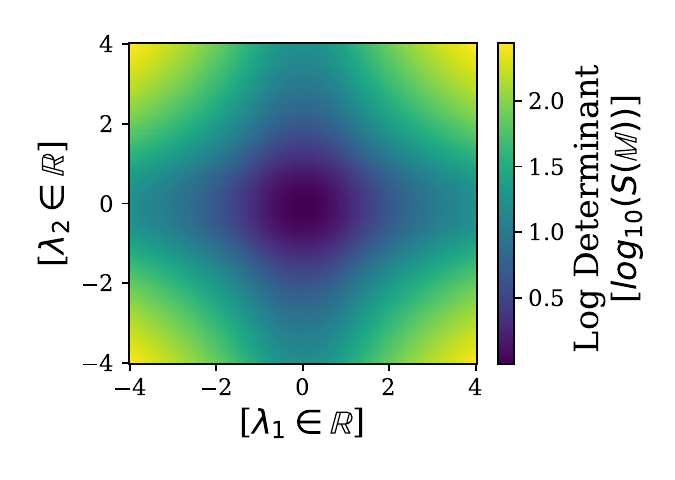}
        \vspace{-9mm}
        \caption{\centering $n = 2$}
        \label{fig:HermMat_2D_Det_2}
    \end{subfigure}    
    \hfill
    \begin{subfigure}[b]{0.48\textwidth}
        \centering
        \includegraphics[width=\textwidth]{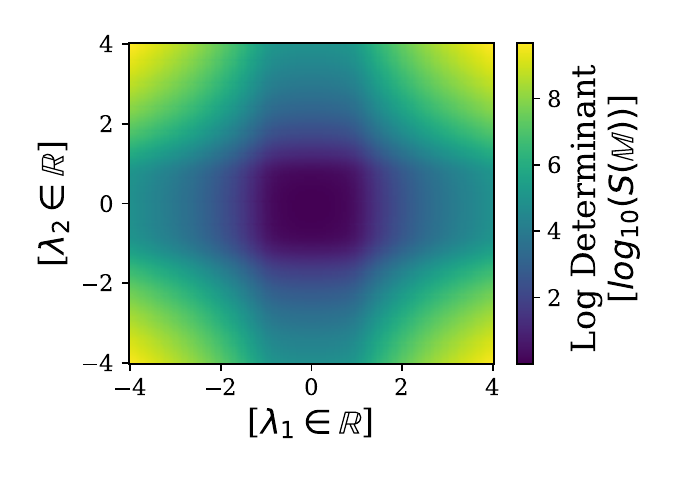}
        \vspace{-9mm}
        \caption{\centering $n = 8$}
        \label{fig:HermMat_2D_Det_8}
    \end{subfigure}
    \vspace{0.1cm}
    \begin{subfigure}[b]{0.48\textwidth}
        \centering
        \includegraphics[width=\textwidth]{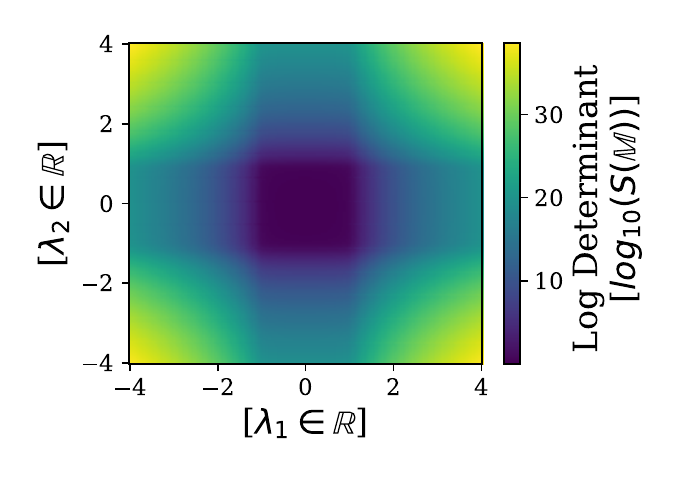}
        \vspace{-9mm}
        \caption{\centering $n = 32$}
        \label{fig:HermMat_2D_Det_32}
    \end{subfigure}
    \hfill
    \begin{subfigure}[b]{0.48\textwidth}
        \centering
        \includegraphics[width=\textwidth]{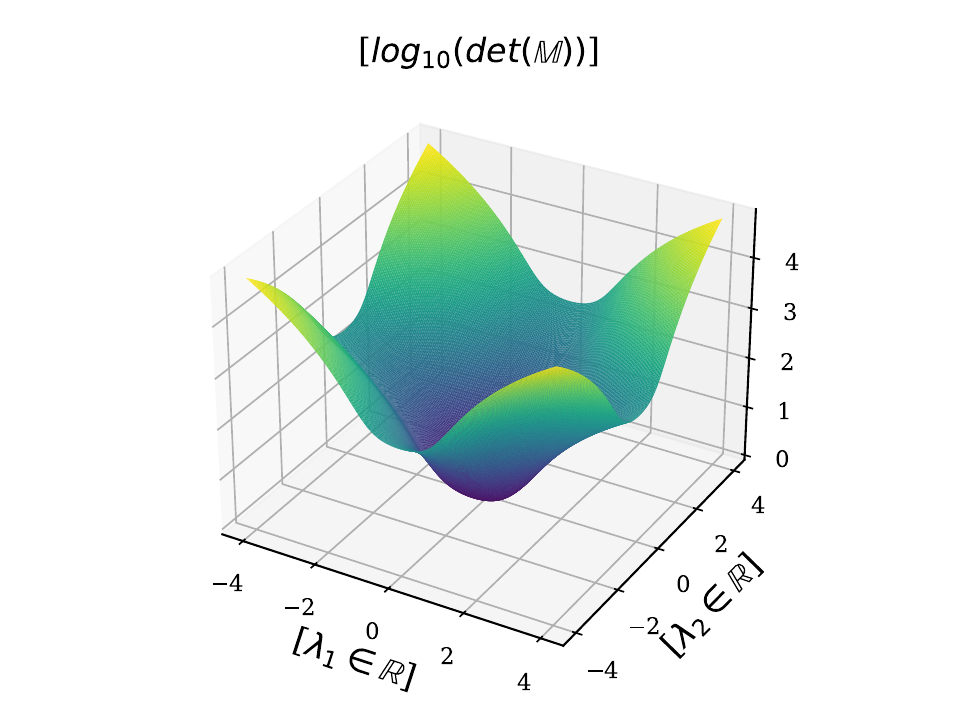}
        \vspace{-9mm}
        \caption{\centering $n = 4$}
        \label{fig:HermMat_2D_Det}
    \end{subfigure}
    \caption{The log of $S(\M_{n,m})$ for Hermitian $\bW \in \C^{2 \times 2}$ is shown for varying $\lambda_1(\bW)$ and $\lambda_2(\bW)$ and for lags (a) $n = 2$, (b) $n = 8$, and (c) $n = 32$.
    The order of magnitude of $S(\M_{n,m})$ scales drastically as $n$ increases, especially when both $|\lambda_1| > 1$ and $|\lambda_2| > 1$, as well as with increasing $n$. 
    When $|\lambda_1| > 1$ and $|\lambda_2| < 1$, or vice versa, $S(\M_{n,m})$ is relatively well-behaved, and the stablest region is when $|\lambda_1| < 1$ and $|\lambda_2| < 1$, the unit square.
    (d) The plot structure in both directions resembles the scalar case \ref{fig:RecMat_Scalar_Dets_Panel}, so, for $m$-dimensional $\bW$, the graph is the scalar case crossed with itself $m$ times.
    }
    \label{fig:HermMat_2D_Cond}
\end{figure}

\begin{remark}
    By the results of Proposition \ref{prop:Block_Mat_Herm_Det_Bds}, $S(\M_{n,m})$ is independent of $n$ provided $|\sigma_{\max}(\bW)| < 1$: $\M_{n,m}$ is most stable when the spectrum of $\bW$ lies within the unit box of dimension $m$.
    In the context of LRNNs, bounding $S(\M_{n,m})$ independent of the number of lags $n$ shows that, although $S(\M_{n,m})$ may increase with additional $m$, there is a finite limit.
    Thus, as $n \to \infty$, this would in theory allow for an infinite look-back at time series without seriously increasing the ill-posedness of the problem. 
\end{remark}

\subsection{Arbitrary Block Case}

Having considered the scalar and block Hermitian cases of the delay matrix $\M_{n,m}$, we can further generalize the results, to a degree, by considering arbitrary weight matrices $\bW$ that need not be diagonalizable. 
Such a generalization would more accurately reflect the behavior of weight matrices of general LRNNs, which often are non-Hermitian and have complex eigenvalues. 
Again,
\begin{align}
    \M_{n,m}: =
    \begin{bmatrix}
        \bI & \bW & \bzero & \cdots & \bzero \\
        \bzero & \bI & \bW & \ddots & \vdots \\
        \vdots & \ddots & \ddots & \ddots & \bzero \\
        \bzero & \cdots & \bzero & \bI & \bW
    \end{bmatrix} \in \C^{m n \times m (n+1)}
    \label{eq:M_general}
\end{align}
and, for convenience of analysis, we consider:
\begin{align}
    \A_{n,m} := \M_{n,m}\M_{n,m}^* = \begin{bmatrix}
        \bI + \bW \bW^* & \bW & \cdots & \bzero \\
        \bW^* & \bI + \bW \bW^* & \ddots & \vdots \\
        \vdots & \ddots & \ddots & \bW \\
        \bzero & \cdots & \bW^* & \bI + \bW \bW^*
    \end{bmatrix} \in \C^{m n \times m n}
    \label{eq:MM_general}
\end{align}
Unlike the scalar and block Hermitian cases that admit explicit results for the singular values of $\M_{n,m}$, the generically non-Hermitian nature of $\bW$ in this generalized setting requires more restrictive assumptions on the spectrum of $\bW$ to arrive at a sufficient condition to prove $\A_{n,m}$ is non-singular.
We first introduce the following results, which will be necessary to prove a sufficient condition for the non-singularity of $\A_{n,m}$.

\begin{definition}\label{def:Block_Irreducibility}
    (Definition 2 ~\cite{Feingold_1962}) The $mn \times mn$ partitioned matrix $\bA = [A_{jk}]$ as in Definition \ref{def:Block_Diag_Dominance} is block irreducible if the $n \times n$ matrix $\bB := (b_{ij} = \|A_{jk}\|)$ for $1 \le j$ and $k \le n$ is irreducible, i.e. the graph of $\bB$ is strongly connected.
\end{definition}

\begin{lemma}\label{lem:Block_Irreducibility}
    (Theorem 1 ~\cite{Feingold_1962}) If the partitioned matrix $\bA$ as assumed in Definition \ref{def:Block_Diag_Dominance} is block irreducible and block diagonally dominant with inequality holding in \eqref{eq:Block_Diag_Dominance} for at least one $j \in \{1,...,n\}$, then $\bA$ is non-singular.
\end{lemma}

\begin{fact}\label{fact:Tridiagonal_Irreducible}
    (~\cite{Horn_Johnson_1985}) Tri-diagonal matrices are irreducible provided all of their super and sub-diagonal entries are nonzero.
\end{fact}

Using the above results, we establish the following theorem.

\begin{theorem}\label{thm:Block_Mat_Gen_Diag_Dom}
    A sufficient condition for $\A_{n,m}$ \eqref{eq:M_general} to be invertible is: 
    \begin{align}\label{eq:Gen_Mat_Weak_Result}
        \sigma_{\min}(\bW) \le \sigma_{\max}(\bW) \le \frac{1}{2}(1 + \sigma_{\min}(\bW)^2).
    \end{align}
    with $\sigma_{\max}(\bW) > 0$.
    When $\sigma_{\min}(\bW) \le \sigma_{\max}(\bW) \le 1$ and $\sigma_{\max}(\bW) > 0$, we have the following improved bound:

    \begin{align}\label{eq:Gen_Mat_Strong_Result_1}
        \left(\frac{\sigma_{\max}(\bW)^3 - \sigma_{\max}(\bW)^2 + 2\sigma_{\max}(\bW) - 1}{\sigma_{\max}(\bW)^2 - \sigma_{\max}(\bW) + 1}\right)^{1/2} \le \sigma_{\min}(\bW), 
    \end{align}
    and when $1 \le \sigma_{\min}(\bW) \le \sigma_{\max}(\bW)$ and $\sigma_{\max}(\bW) > 0$, we have the improved bound:
    \begin{align}\label{eq:Gen_Mat_Strong_Result_2}
        \sigma_{\max}(\bW) \le \sigma_{\min}(\bW)^2 - \sigma_{\min}(\bW) + 1. 
    \end{align}
\end{theorem}

\begin{proof}
    $\bW \in \C^{m \times m}$ has an SVD, $\bW = \bU \bSigma \bV^*$, where $\bU$ and $\bV$ are unitary and $\bSigma$ contains the singular values of $\bW$.
    To find a condition for $\M_{n,m}$ to be non-singular, we enforce strict block diagonal dominance from Definition \ref{def:Block_Diag_Dominance}. 
    Since $\bW \bW^*$ is positive semi-definite, adding $\bI$ to it results in a positive definite--and thus invertible--matrix:
    \begin{align}
    \begin{split}
        (\bI + \bW \bW^*)^{-1} &= (\bU \bI \bU^* + (\bU \bSigma \bV^*)(\bU \bSigma \bV^*)^*)^{-1}
        = \bU(\bI + \bSigma^2)^{-1} \bU^*
    \end{split}
    \end{align}
    With $\M_{n,m}$ being a tri-diagonal matrix, the Gerschgorin set associated with each block of rows requires only two components, namely $(\bI + \bW \bW^*)^{-1}\bW$ and $(\bI + \bW \bW^*)^{-1}\bW^*$.
    Since the operator norm is unitarily invariant, for the first component,
    \begin{align}
    \begin{split}
        \|(\bI + \bW \bW^*)^{-1} \bW\|_{\op} &= \|\bU(\bI + \bSigma^2)^{-1} \bU^* \bU \bSigma \bV^*\|_{\op} \\
        &= \|(\bI + \bSigma^2)^{-1} \bSigma \|_{\op} \\
        &= \max_{j = \{1,.., m\}} \left(\frac{\sigma_j(\bW)}{1 + \sigma_j(\bW)^2}\right) \le \frac{1}{2}
    \end{split}
    \label{eq:MM_block_diag_term_1}
    \end{align}
    The upper bound of $\frac{1}{2}$ is attained when $\sigma_j(\bW) = 1$ for at least one $j \in \{1,.., m\}$. 
    For the second component, again making use of the unitary invariance of the norm,
    \begin{align}
    \begin{split}
        \|(\bI + \bW \bW^*)^{-1} \bW^*\|_{\op} &= \|\bU(\bI + \bSigma^2)^{-1} \bU^* (\bV^*)^* \bSigma^* \bU^*\|_{\op} \\
        &= \|(\bI + \bSigma^2)^{-1} \bU^* \bV \bSigma\|_{\op} \\
        &\le \|(\bI + \bSigma^2)^{-1}\|_{\op} \|\bU^* \bV \bSigma\|_{\op} \\
        &= \max_{j = \{1,.., m\}} \left(\frac{1}{1 + \sigma_j(\bW)^2}\right) \|\bSigma\|_{\op} \\
        &= \frac{\sigma_{\max}(\bW)}{1 + \sigma_{\min}(\bW)^2}
    \end{split}
    \label{eq:MM_block_diag_term_2}
    \end{align}
    For $\A_{n,m}$ to be strictly block diagonally dominant, which is a \textit{sufficient} condition to guarantee that $\A_{n,m}$ is non-singular, we require:
    \begin{align}\label{eq:MM_block_diag_dom_cond}
        \|(\bI + \bW \bW^*)^{-1} \bW\|_{\op} + \|(\bI + \bW \bW^*)^{-1} \bW^*\|_{\op} \le C < 1
    \end{align}
    for some constant $C$. 
    Using the bounds derived in \eqref{eq:MM_block_diag_term_1} and \eqref{eq:MM_block_diag_term_2} and substituting them into \eqref{eq:MM_block_diag_dom_cond}, we pick $C$ to be:
    \begin{align}
        C= \frac{1}{2} + \frac{\sigma_{\max}(\bW)}{1 + \sigma_{\min}(\bW)^2} < 1
    \end{align}
    This provides a condition on the relationship between $\sigma_{\max}(\bW)$ and $\sigma_{\min}(\bW)$,
    \begin{align}\label{eq:MM_Gen_Diag_Dom_Cond_1}
        \sigma_{\min}(\bW) \le \sigma_{\max}(\bW) < \frac{1}{2}(1 + \sigma_{\min}(\bW)^2)
    \end{align}
    The above bound is not tight. 
    To refine it, note that for $0 \le \sigma_{\min}(\bW) < 1$, $0 \le \sigma_{\max}(\bW) < 1$.
    And, for $\sigma_{\min}(\bW) > 1$, $\sigma_{\max}(\bW) > 1$ as well, and $\M_{n,m}$ is strictly diagonally dominant. 
    Although we could also consider cases where $\sigma_{\min}(\bW) < 1$ and $\sigma_{\max}(\bW) > 1$ with some $\sigma_j(\bW)$ for $j \in \{2,..,m-1\}$ being the singular value closest to $1$, for simplicity we consider the cases when $\sigma_{\min}(\bW) \le \sigma_{\max}(\bW) < 1$ and when $\sigma_{\max}(\bW) \ge \sigma_{\min}(\bW) > 1$.

    \textbf{Case 1:} For $\sigma_{\min}(\bW) \le \sigma_{\max}(\bW) < 1$, $\sigma_{\max}(\bW) = 1$ is the singular value that maximizes \eqref{eq:MM_block_diag_term_1}, so we can improve the $\frac{1}{2}$ bound for the first component, with the second as is:
    \begin{align*}
        \frac{\sigma_{\max}(\bW)}{1 + \sigma_{\max}(\bW)^2} + \frac{\sigma_{\max}(\bW)}{1 + \sigma_{\min}(\bW)^2} < 1
    \end{align*}
    Re-arranging provides a constraint condition on $\sigma_{\max}(\bW)$:
    \begin{align}\label{eq:MM_Gen_Diag_Dom_Cond_2}
        \left(\frac{\sigma_{\max}(\bW)^3 - \sigma_{\max}(\bW)^2 + 2\sigma_{\max}(\bW) - 1}{\sigma_{\max}(\bW)^2 - \sigma_{\max}(\bW) + 1}\right)^{1/2} < \sigma_{\min}(\bW) 
    \end{align}

    \textbf{Case 2:} When $\sigma_{\max}(\bW) \ge \sigma_{\min}(\bW) > 1$, $\sigma_{\min}(\bW)$ is the singular value that maximizes \eqref{eq:MM_block_diag_term_1}, and keeping \eqref{eq:MM_block_diag_term_2} as is: 
    \begin{align*}
        \frac{\sigma_{\min}(\bW)}{1 + \sigma_{\min}(\bW)^2} + \frac{\sigma_{\max}(\bW)}{1+\sigma_{\min}(\bW)^2} < 1
    \end{align*}
    is the condition that, when re-arranged, leads to
    \begin{align}\label{eq:MM_Gen_Diag_Dom_Cond_3}
        \sigma_{\min}(\bW) \le \sigma_{\max}(\bW) < \sigma_{\min}(\bW)^2 - \sigma_{\min}(\bW) + 1
    \end{align}

    The above conditions \eqref{eq:MM_Gen_Diag_Dom_Cond_1}, \eqref{eq:MM_Gen_Diag_Dom_Cond_2}, and \eqref{eq:MM_Gen_Diag_Dom_Cond_3} are strict inequalities. 
    To establish a sharper bound by showing they hold under equality too, we use Lemma \ref{lem:Block_Irreducibility}.
    $\A_{n,m}$ is block irreducible if the matrix $\bB$ the entries of which are the norms of each of the blocks of $\A_{n,m}$ is irreducible as in Definition \ref{def:Block_Irreducibility}. 
    Since $\A_{n,m}$ is a block tri-diagonal matrix, $\bB$ is a scalar tri-diagonal matrix. 
    By Fact \ref{fact:Tridiagonal_Irreducible}, $\bB$ is irreducible if none of its super or sub-diagonal entries are zero, which is true when $\|\bW\|_{\op} \neq 0$. 
    By assumption, $\sigma_{\max}(\bW) > 0$, so $\|\bW\|_{\op} \neq 0$, and thus $\A_{n,m}$ is block irreducible. 
    (If $\bW = \bzero$, then $\A_{n,m}$ reduces to the identity matrix and is full-rank and well-conditioned.) 

    Provided \eqref{eq:MM_Gen_Diag_Dom_Cond_1}, \eqref{eq:MM_Gen_Diag_Dom_Cond_2}, or \eqref{eq:MM_Gen_Diag_Dom_Cond_3} hold under the relevant range restrictions on the spectrum of $\bW$, then $\A_{n,m}$ is block diagonally dominant, and inequality is achieved when $j = 1$, since diagonal dominance requires $\|(\bI + \bW \bW^*)^{-1} \bW \|_{\op} \le 1$ when $j = 1$, and by \eqref{eq:MM_block_diag_term_1} it is always less than $\frac{1}{2}$. 
    And, by Lemma \ref{lem:Block_Irreducibility}, $\A_{n,m}$ is non-singular when strengthening \eqref{eq:MM_Gen_Diag_Dom_Cond_1}, \eqref{eq:MM_Gen_Diag_Dom_Cond_2}, and \eqref{eq:MM_Gen_Diag_Dom_Cond_3} from strict inequality to equality.
\end{proof}

In the scalar and Hermitian cases, $\kappa(\M_{n,m})$ was small when the eigenvalues of $\bW$ were within the unit circle. 
Even so, for $\sigma_{\max}(\bW)$ near $1$, $\kappa(\M_{n,m})$ could grow exponentially large, and thus the best distribution of the singular values of $\bW$ would be in a small disc within the unit circle. 
For general $\bW$, the conditions that Theorem \ref{thm:Block_Mat_Gen_Diag_Dom} provide agree with the observation that $\M_{n,m}$ is non-singular, and thus of a finite condition number, when $\sigma_{\max}(\bW)$ is contained well-within the unit circle.

\begin{remark}
    The conditions in Theorem \ref{thm:Block_Mat_Gen_Diag_Dom}, particularly \eqref{eq:MM_Gen_Diag_Dom_Cond_2} and \eqref{eq:MM_Gen_Diag_Dom_Cond_3}, are sufficient to ensure $\A_{n,m}$ is non-singular.
    Any improvement in the bounds will come from establishing a tighter bound on the term $\|(\bI + \bW \bW^*)^{-1} \bW^*\|$ as in \eqref{eq:MM_block_diag_term_2}, which would require knowledge of additional singular values, $\sigma_j(\bW)$ for $j \in \{2,...,n-1\}$, and the orientation of the basis vectors $\bU$ in relation to the canonical basis of $\bI$, or using an alternative method that admits a tighter eigenvalue inclusion set for $\A_{n,m}$.
\end{remark}

\begin{figure}
    \centering
    \begin{subfigure}{.33\textwidth}
        \centering
        \includegraphics[width=\linewidth]{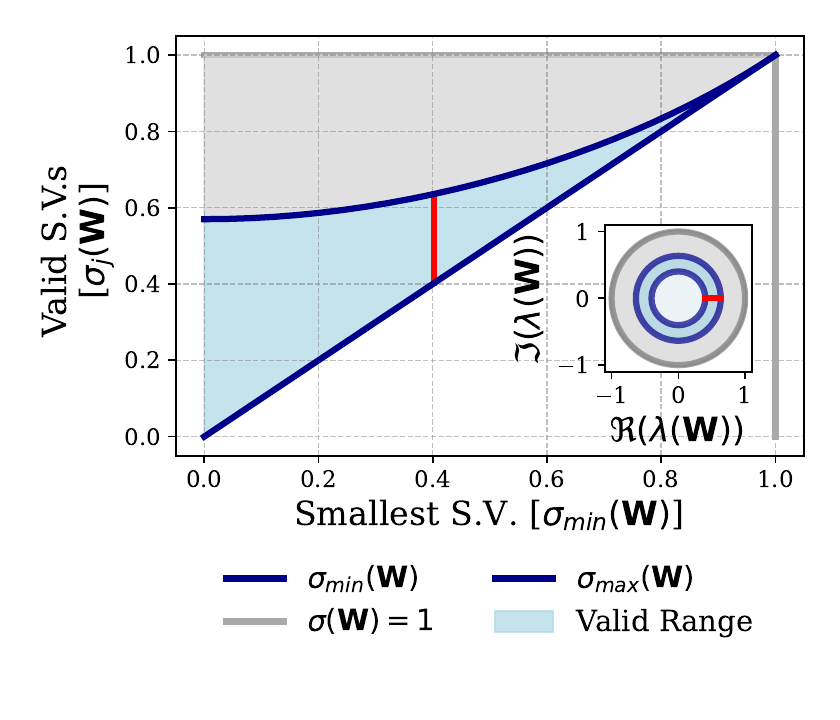}
        \vspace{-1cm}
        \caption{\centering}
        \label{fig:RecMat_Gen_Sing_Cond}
    \end{subfigure}%
    \begin{subfigure}{.33\textwidth}
        \centering
        \includegraphics[width=\linewidth]{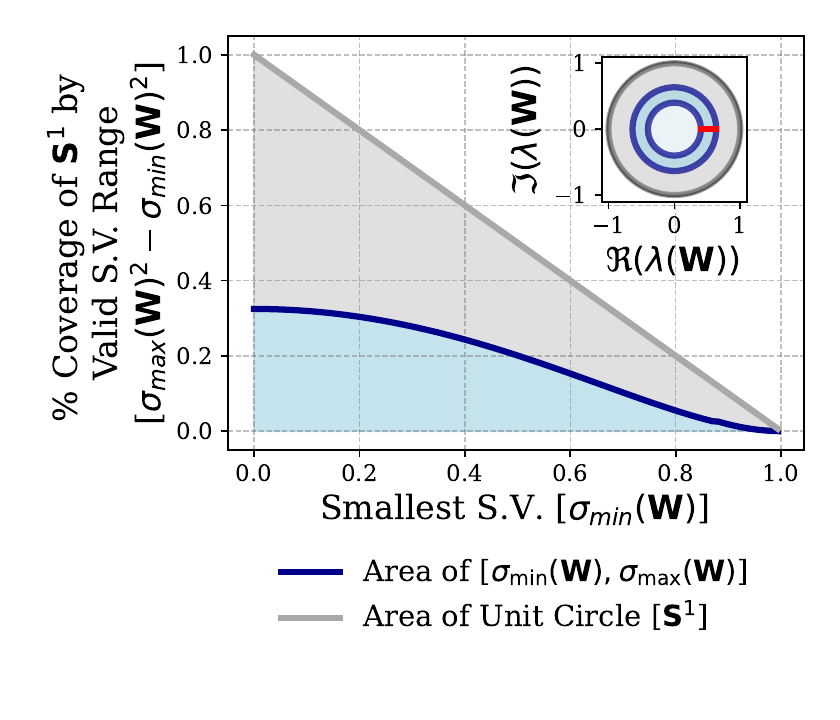}
        \vspace{-1cm}
        \caption{\centering}
        \label{fig:RecMat_Gen_Sing_Area}
        \end{subfigure}
    \begin{subfigure}{.33\textwidth}
        \centering
        \includegraphics[width=\linewidth]{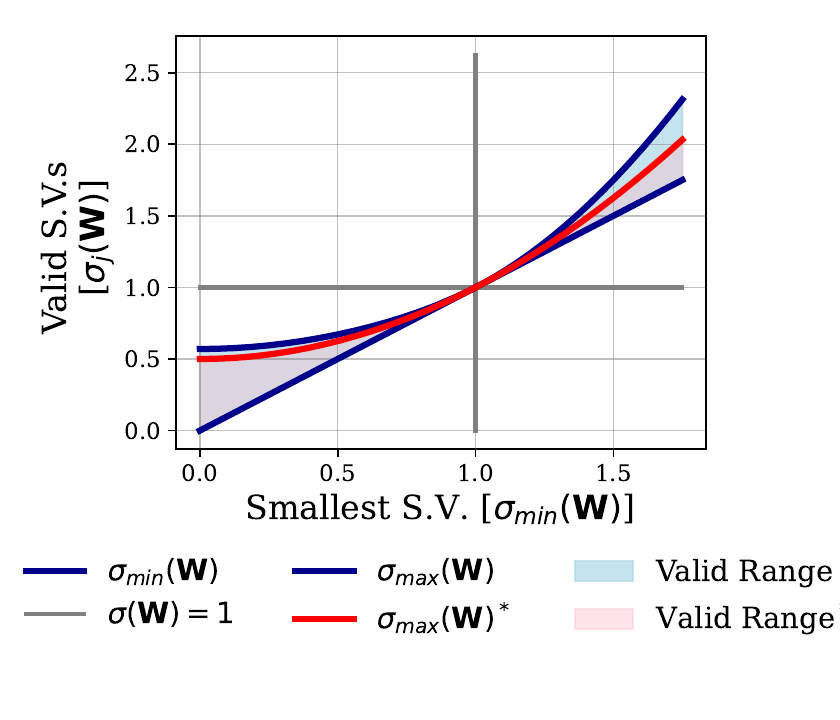}
        \vspace{-1cm}
        \caption{\centering}
        \label{fig:RecMat_Gen_Sing_Cond2}
    \end{subfigure}
    \caption{(a) The admissible spectrum of $\bW$ lies within the blue region ranging from the dark blue lines $\sigma_{\min}(\bW)$ and $\sigma_{\max}(\bW)$. 
    The area shaded in gray, which \textit{is} admissible in the scalar and block Hermitian cases, represents the \textit{potential} area of improved coverage if a tighter bound on the block diagonal dominance condition is even possible.
    The red line provides an example range from $\sigma_{\min}(\bW)$ to $\sigma_{\max}(\bW)$ and the toroidal section of the unit circle it defines in which the spectrum of $\bW$ can reside for $\M_n$ to be full row rank and thus an embedding.
    (b) The associated \textit{area} of the torus defined by the admissible singular value range relative to the area of the unit circle.
    (c) The admissible singular value range of $\bW$ extended beyond $\sigma_{\min}(\bW) \in [0,1]$ to $[0,2]$, with the weaker result \eqref{eq:Gen_Mat_Weak_Result} in red and the stronger results \eqref{eq:Gen_Mat_Strong_Result_1} and \eqref{eq:Gen_Mat_Strong_Result_2} in blue.
    The ``pinch" point where $\sigma_{\max}(\bW), \sigma_{\min}(\bW) \to 1$ is when $\bW$ is unitary.}
    \label{fig:RecMat_Gen_Sings}
\end{figure}

\begin{remark}
    If $\bU = \bV$, then $\bU \bV^* = \bI$, meaning $\bW$ is Hermitian and $\M_{n,m}$ is strictly block diagonal dominant when $\sigma_{\max}(\bW) \neq 1$, which agrees with prior results. 
\end{remark}

\subsection{Unitary Weight Matrices}

A special case for non-Hermitian $\bW$ is when $\bW$ is unitary. 
In machine learning applications, RNNs of the form \eqref{eq:LRNN} with weight matrices $\bW$ the eigenvalues of which deviate from absolute value $1$ suffer from exploding and vanishing gradient issues that pose difficulties for learning long-term dependencies from the data. 
One approach to address these issues is to enforce $\bW$ to be unitary, which has been shown to work well, at least for nonlinear RNN ~\cite{Arjovsky_2016}. 

For LRNNs, unitary $\bW$ are viable, but potentially unstable. 
In effect, unitary matrices perform a rotation of coordinates with each iterate of the recurrence relation. 
The unitary matrix is shown in Figure \ref{fig:RecMat_Gen_Sing_Cond2} as the ``pinch point" in the admissible range ($[\sigma_{\min}(\bW), \sigma_{\max}(\bW)]$) of $\bW$, where $\sigma_{\min}(\bW), \sigma_{\max}(\bW) \to 1$, and thus $\sigma_j(\bW) \to 1$ for all $j \in \{1,...,n\}$. 
The conditions from Theorem \ref{thm:Block_Mat_Gen_Diag_Dom} indicate that multivariate delay-coordinate embeddings of time series using LRNNs of the form \eqref{eq:LRNN} are always embeddings for any choice of unitary matrix $\bW$. 
Such embeddings, however, may not be stable as $\M_{n,m}$, despite being full row rank, may be ill-conditioned.

\subsubsection{Singular Values}

Bounding the singular values of $\M_{n,m}$ with an arbitrary $\bW$ is more difficult. 
However, with certain assumptions, we can generate a few helpful results. 
Often, the smallest singular value is the most difficult to bound. 
Yet, for the largest singular value, we can establish the following general bound.

\begin{proposition}\label{prop:Block_Mat_Gen_S_Max}
    Let $\M_{n,m}$ be as in \eqref{eq:M_general}. 
    Then, for any $\sigma_{\max}(\bW) \ge 0$, we have $\sigma_{\max}(\M_{n,m}) \le \sigma_{\max}(\bW) + 1$.
\end{proposition}

\begin{proof}
    Using Theorem \ref{thm:Block_Gerschgorin}, the eigenvalues of $\A_{n,m}$ are contained in the region:
    \begin{align*}
        1 &\le \|(\bI + \bW \bW^* - \lambda(\A_{n,m}) \bI)^{-1} \bW\|_{\op} + \|(\bI + \bW \bW^* - \lambda(\A_{n,m}) \bI)^{-1} \bW^*\|_{\op} \\
        &\le \frac{2 \sigma_{\max}(\bW)}{\min_{j = \{1,.., m\}} |\lambda(\A_{n,m}) - 1 + \sigma_j(\bW)^2|}
    \end{align*}
    Re-arranging, we have:
    \begin{align*}
        \min_{j = \{1,.., m\}} \left|\lambda(\A_{n,m}) - (1 + \sigma_j(\bW)^2) \right| \le 2 \sigma_{\max}(\bW)
    \end{align*}
    Suppose that $\sigma_c(\bW)$ is the minimizer for $c \in \{1,...,m\}$. 
    Then,
    \begin{align*}
        \sigma_c(\bW)^2 -2 \sigma_{\max}(\bW) + 1 \le \lambda(\A_{n,m}) \le \sigma_c(\bW)^2 + 2 \sigma_{\max}(\bW) + 1
    \end{align*}
    Since $\A_{n,m}$ is Hermitian, by the spectral theorem it has all real eigenvalues. 
    As a Hermitian product, it is positive semi-definite, so all the eigenvalues are non-negative.
    For an upper bound,
    \begin{align*}
        \lambda_{\max}(\A_{n,m}) \le \sigma_c(\bW)^2 + 2 \sigma_{\max}(\bW) + 1 
        \le (\sigma_{\max}(\bW) + 1)^2
    \end{align*}
    Thus, $\sigma_{\max}(\M_{n,m}) \le \sigma_{\max}(\bW) + 1$, so $\sigma_{\max}(\M_{n,m})$ scales independently of $m$ and $n$ for viable $\sigma_{\max}(\bW)$.
\end{proof}

\begin{remark}
    This bound corresponds with the bound for $\sigma_{\max}(\bW)$ for both the scalar case (see Proposition \ref{prop:Scalar_Mat_Cond}) and Hermitian case (see Theorem \ref{thm:Block_Mat_Herm_Cond}).
\end{remark}

In the context of generating well-conditioned LRNN delay-coordinate maps, we have seen in the case of Hermitian $\bW$ that eigenvalues in the unit circle tend to produce smaller condition numbers and determinants and thus more stable maps.
It is reasonable then, in the general case, to consider eigenvalues of $\bW$ in the unit circle, i.e. $\sigma_{\max}(\bW) < 1$. 
In this direction, we can arrive at the following, more restrictive, but still informative result.

\begin{proposition}
    \label{prop:Block_Mat_Gen_Cond_Bds}
    Let $\A_{n,m}$ be strictly row block strictly diagonally dominant. 
    For $\sigma_{\max}(\bW) < \frac{1}{2}$,
    \begin{align}\label{eq:RecMat_Gen_Cond_Bds}
        \kappa(\M_{n,m}) \le \left(\frac{1 + 2 \sigma_{\max}(\bW) + \sigma_{\max}(\bW)^2}{1 - 2 \sigma_{\max}(\bW) + \sigma_{\min}(\bW)^2}\right)^{{1/2}}
    \end{align}
\end{proposition}

\begin{proof}
    If $\sigma_{\max}(\bW) < \frac{1}{2}$, then $\A_{n,m}$ is strictly block diagonally dominant and thus non-singular. 
    Consider the block Gerschgorin eigenvalue inclusion set from Theorem \ref{thm:Block_Gerschgorin} that contains $\lambda$, the eigenvalue of $\A_{n,m}$:
    \begin{align*}
        R_j = \sum_{k = 1, k \neq j}^n \|(A_{jj} - \lambda I)^{-1} A_{jk}\|_{\op} \ge 1
    \end{align*}
    In the worst case, when $j \in \{2,...,n-1\}$,
    \begin{align*}
        1 &\le \|((1-\lambda)\bI + \bSigma^2)^{-1} \bSigma\|_{\op} + \|((1-\lambda)\bI + \bSigma^2)^{-1}\| \|\bSigma\|_{\op} \\
        &= \max_{j \in \{1,.., m\}} \left( \frac{\sigma_j(\bW)}{|1 - \lambda + \sigma_j(\bW)^2|} \right) + \sigma_{\max}(\bW) \max_{j \in \{1,.., m\}} \left( \frac{1}{|1 - \lambda + \sigma_j(\bW)^2|} \right) \\
        &\le \frac{2 \sigma_{\max}(\bW)}{\min_{j \in \{1,.., m\}} |\lambda - (1 + \sigma_j(\bW)^2)|}
    \end{align*}
    The above expression will define a lower bound on $\lambda_{\min}:=\lambda_{\min}(\A_{n,m})$, and an upper bound on $\lambda_{\max}:= \lambda_{\max}(\A_{n,m})$. 
    Since $\A_{n,m}$ is Hermitian, its eigenvalues are real, so $\lambda_{\min}, \lambda_{\max} \in \R$, and thus:
    \begin{align*}
        -2 \sigma_{\max}(\bW) \le \min_{j \in \{1,.., m\}} |\lambda - (1 + \sigma_j(\bW)^2)| \le 2 \sigma_{\max}(\bW)
    \end{align*}
    where the minimizer $j$ depends on $\lambda$.
    Recall $\sigma_{j}(\M_{n,m}) = \sqrt{\lambda_j(\A_{n,m})}$. 
    For $\lambda_{\max}$, $j = 1$ is the minimizer, and an upper bound is:
    \begin{align*}
        \sigma_{\max}(\M_{n,m}) \le 1 + \sigma_{\max}(\bW)
    \end{align*}
    For $\lambda_{\min}$, $j = m$ is the minimizer, leading to the lower bound:
    \begin{align*}
        \sigma_{\min}(\M_{n,m}) \ge (1 + \sigma_{\min}(\bW)^2 - 2 \sigma_{\max}(\bW))^{1/2}
    \end{align*}
    Combining the above two bounds leads to a bound on  $\kappa(\M_{n,m})$.
\end{proof}

\begin{remark}
   Importantly, the above bound shows the delay-matrix condition number, $\kappa(\M_{n,m})$, grows independently of the size of the weight matrix $\bW$ and the number of lags $n$, at least for $\sigma_{\max}(\bW) \le \frac{1}{2}$, implying additional lags or hidden states will mean $\kappa(\M_{n,m})$ is still bounded.
   While the bound is relatively good for small $\sigma_{\max}(\bW)$, agreeing as it should when $\bW = 0$ and thus $\sigma_{\max}(\bW) = 0$ so that $\A_{n,m}$ is the identity with $\kappa(\bM_n) = 1$, as $\sigma_{\max}(\bW) \to \frac{1}{2}$, the bound becomes unbounded. 
   This seems to agree with the unbounded behavior of the condition number in the scalar and Hermitian cases for eigenvalues approaching a magnitude of $1$.
\end{remark}

\begin{remark}
    The bound \eqref{eq:RecMat_Gen_Cond_Bds} in Proposition \ref{prop:Block_Mat_Gen_Cond_Bds} uses \eqref{eq:MM_Gen_Diag_Dom_Cond_1}, a bound that is looser than \eqref{eq:MM_Gen_Diag_Dom_Cond_3}. 
    Using \eqref{eq:MM_Gen_Diag_Dom_Cond_3} could offer a better bound. 
\end{remark}

\begin{remark}
    In the Hermitian case, $\bW$ being diagonalizable meant $\A_{m,n}$ could be turned into a series of unitary matrices multiplied by an ultimately diagonal matrix (see Remark \ref{rmk:Herm_Mat_Factor}, essentially indicating that the dynamics could become completely decoupled in some appropriate coordinate frame via rotation.
    Furthermore, that $\bW$ is Hermitian implies, by the spectral theorem, that it possesses all real eigenvalues, and therefore can only introduce growth correlations of exponential growth and decay between current and past iterates in the LRNN.
    In the general case, $\bW$ may have complex eigenvalues and cannot in general be factored into a series of unitary matrices multiplied by a diagonal matrix; thus, the dynamics may typically be coupled to each other, and the presence of complex eigenvalues in $\bW$ may allow more expressive oscillatory behavior relating present to past states.
\end{remark}

\begin{figure}
    \centering
    \begin{subfigure}{.33\textwidth}
        \centering
        \includegraphics[width=0.95\linewidth]{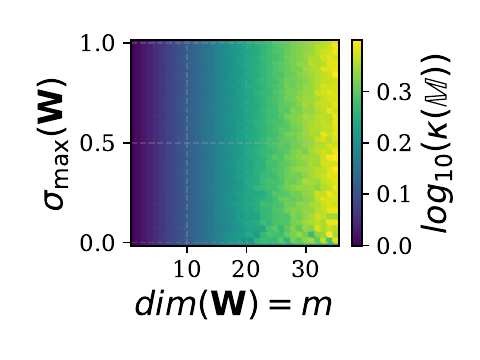}
        \vspace{-5mm}
        \caption{\centering $n = 2$}
        \label{fig:RecMat_Gen_Cond_2}
    \end{subfigure}%
    \begin{subfigure}{.33\textwidth}
        \centering
        \includegraphics[width=0.95\linewidth]{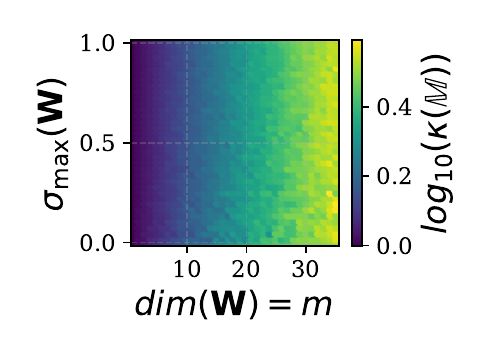}
        \vspace{-5mm}
        \caption{\centering $n = 8$}
        \label{fig:RecMat_Gen_Cond_8}
    \end{subfigure}%
    \begin{subfigure}{.33\textwidth}
        \centering
        \includegraphics[width=0.95\linewidth]{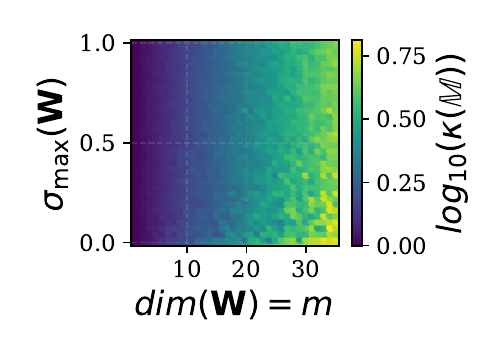}
        \vspace{-5mm}
        \caption{\centering $n = 16$}
        \label{fig:RecMat_Gen_Cond_16}
    \end{subfigure}
    \caption{Log condition numbers $\kappa(\M_{n,m})$, where the $\M_{n,m}$ have general $\bW$, are plotted against a continuum of choices of bounds on the maximal singular value of $\bW$, $\sigma_{\max}(\bW)$, as well as for various choices of sizing $m$ for $\bW$. 
    The sub-figures depict the results for (a) $n = 4$, (b) $n = 8$, and (c) $n = 16$ lags. 
    Increasing $m$ increases the ill-conditioning of $\M_{n,m}$. Note the difference in colorbar scaling.}
    \label{fig:RecMat_Gen_Cond}
\end{figure}

\begin{figure}
    \centering
    \begin{subfigure}{.33\textwidth}
        \centering
        \includegraphics[width=.95\linewidth]{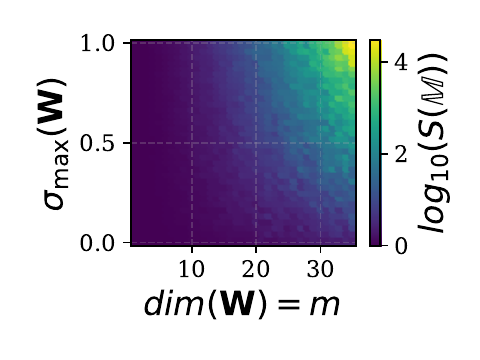}
        \vspace{-5mm}
        \caption{\centering $n = 2$}
        \label{fig:RecMat_Gen_Det_2}
    \end{subfigure}%
    \begin{subfigure}{.33\textwidth}
        \centering
        \includegraphics[width=.95\linewidth]{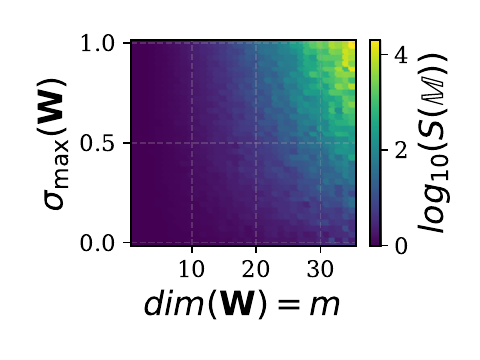}
        \vspace{-5mm}
        \caption{\centering $n = 8$}
        \label{fig:RecMat_Gen_Det_8}
    \end{subfigure}%
    \begin{subfigure}{.33\textwidth}
        \centering
        \includegraphics[width=.95\linewidth]{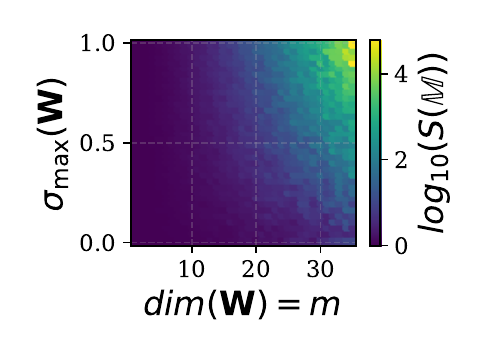}
        \vspace{-5mm}
        \caption{\centering $n = 16$}
        \label{fig:RecMat_Gen_Det_16}
    \end{subfigure}
    \caption{Log values of the generalized determinant $S(\M_{n,m})$, for $\M_{n,m}$ with arbitrary $\bW$, are plotted against varying choices of $\sigma_{\max}(\bW)$ and for different sizes $m$ of $\bW$, with the sub-figures showing: (a) $n = 4$, (b) $n = 8$, and (c) $n = 16$ delays. Note the slight difference colorbar scaling.}
    \label{fig:RecMat_Gen_Det}
\end{figure}

\begin{figure}
    \centering
    \begin{subfigure}{.5\textwidth}
        \centering
        \includegraphics[width=.95\linewidth]{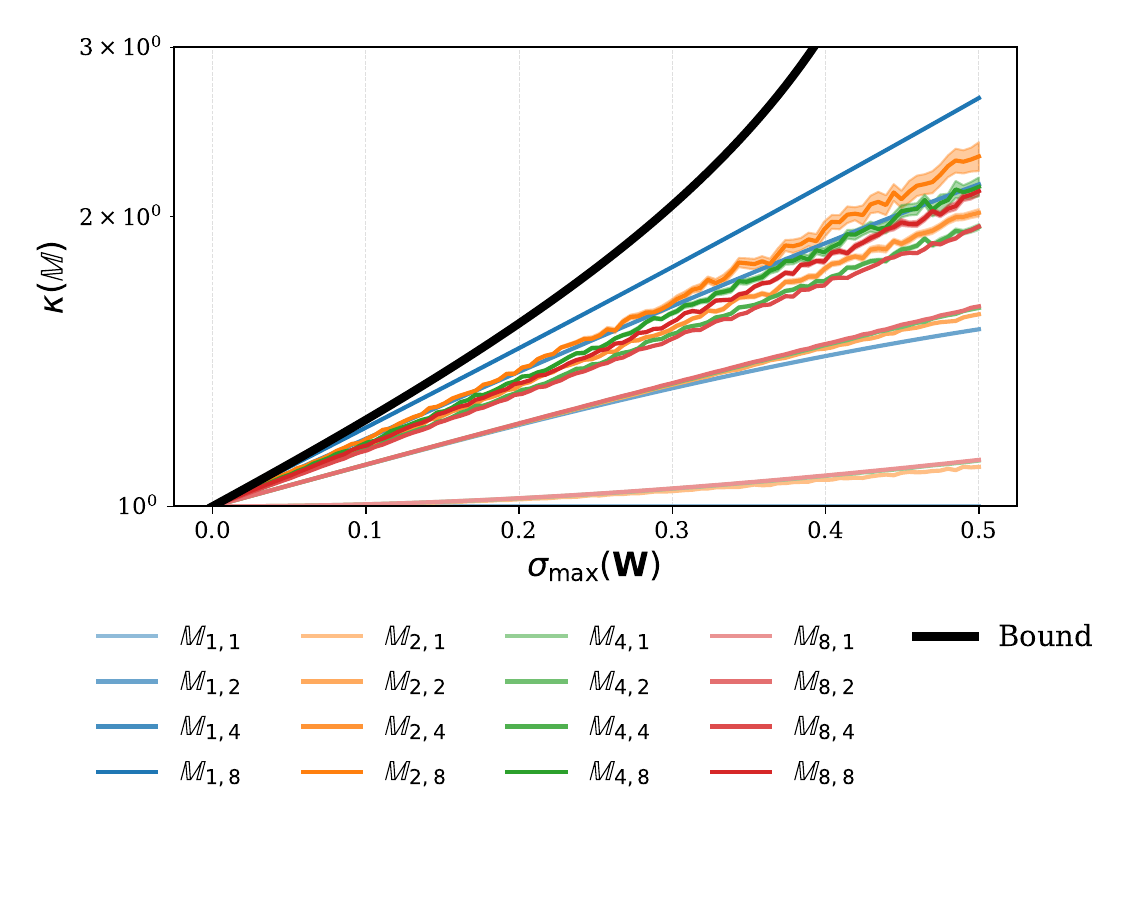}
        \vspace{-7mm}
        \caption{\centering}
        \label{fig:RecMat_Gen_Bound_Test}
    \end{subfigure}%
    \begin{subfigure}{.5\textwidth}
        \centering
        \includegraphics[width=.95\linewidth]{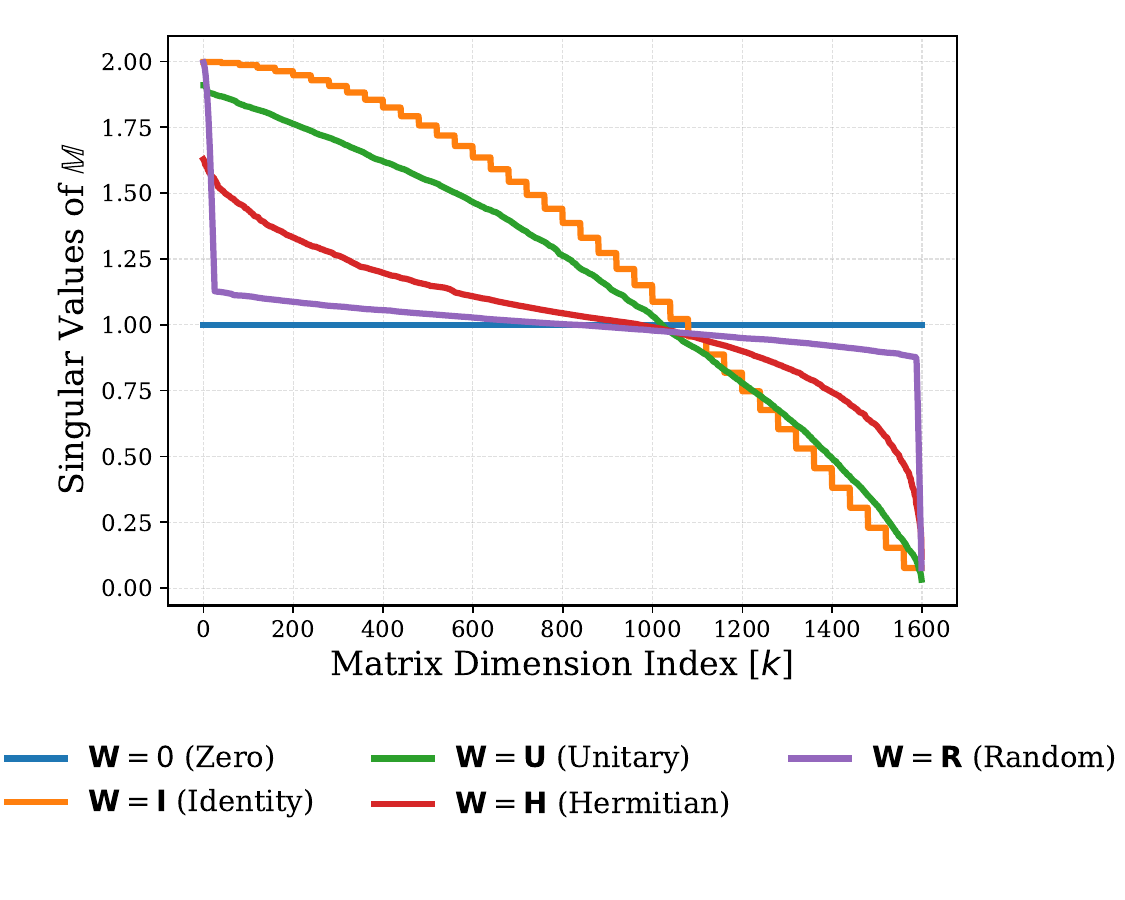}
        \vspace{-7mm}
        \caption{\centering}
        \label{fig:RecMat_Gen_Bound_Test_1}
    \end{subfigure}
    \caption{(a) shows in black the bound on $\kappa(\M_{n,m})$ from Proposition \ref{prop:Block_Mat_Gen_Cond_Bds} that is independent of $m$ and $n$, along with $\kappa(\M_{n,m})$ for various numerically-simulated $\M_{n,m}$ with $\bW$ of varying dimensions $m$ and lags $n$, plotted against $\sigma_{\max}(\bW)$. 
    While the bound is unbounded on a log scale as $\sigma_{\max}(\bW) \to 1$, all the numerically-simulated results grow roughly linearly in the log-scale.
    Increasing $m$ reduces $\kappa(\M_{n,m})$ for a fixed $\sigma_{\max}(\bW)$, while increasing $n$ increases $\kappa(\M_{n,m})$.
    (b) depicts a comparison of the singular value distributions of $\M_{n,m}$ for various classes of $\bW$. 
    The distribution for the random $\bW$, which are almost always full rank, roughly approaches the distribution for when $\bW = \bzero$, and the distribution for unitary $\bW$ is similar to that of identity $\bW$, as both types of $\bW$ have unit-magnitude eigenvalues. 
    When $\bW = \bzero$ or $\bW$ is random, the spectrum of $\M_{n,m}$ approaches unity, in which case $\M_{n,m}$ behaves like a unitary transformation from $\Psi(\bh_k)$ into $\Phi_{\bx_{k-1}}$, meaning it is almost an isometric embedding.}
    \label{fig:RecMat_Gen_Bound_Test1}
\end{figure}

To numerically characterize the condition number (Figure \ref{fig:RecMat_Gen_Cond} and determinant (Figure \ref{fig:RecMat_Gen_Det}) of $\M_{n,m}$ for LRNNs with arbitrary $\bW$, we generated random $\bW$ with spectral norm $\sigma_{\max}(\bW) \in [0,1]$, a suitable range based on Theorem \ref{thm:Block_Mat_Gen_Diag_Dom}.
We swept across $m \in \{1,...,35\}$ and $35$ evenly-spaced values of $\sigma_{\max}(\bW)$ between $[0,1]$.
For each iterate, we generated $\bW$ by initializing a matrix with normally-distributed, mean-zero, unit-variance entries, dividing it by its spectral norm, and rescaling it by the chosen $\sigma_{\max}(\bW)$.
We then computed the condition number and generalized determinant of $\M_{n,m}$.
Figure \ref{fig:RecMat_Gen_Cond} shows that increasing $m$ tends to increase the condition number, but the condition number does not vary much with $\sigma_{\max}(\bW)$.
Figure \ref{fig:RecMat_Gen_Det} shows the determinant is large when both $\sigma_{\max}(\bW)$ and $m$ are large, and small when $\sigma_{\max}(\bW)$ is small and $m$ is large or when $\sigma_{\max}(\bW)$ is large and $m$ is small.

\subsection{Increasing Lags}

Increasing the number of lags in the delay matrix $\M_{n,m}$ with a particular $\bW$ can only increase the ill-conditioning of $\M_{n,m}$.
To show this, we begin by referencing the following eigenvalue interlacing theorem.

\begin{theorem}\label{thm:Cauchy_Interlacing}
    (Theorem 1 ~\cite{Hwang_2004}) Let $\bH \in \C^{n \times n}$ be a Hermitian matrix with real eigenvalues $\lambda_1 \le ... \le \lambda_n$, partitioned as
    \begin{align*}
        \bH = \begin{bmatrix}
            \bA & \bB^* \\
            \bB & \bC
        \end{bmatrix}
    \end{align*}
    where $\bA \in \C^{m \times m}$, $\bB \in \C^{(n-m) \times m}$ and $\bC \in \C^{(n-m) \times (n-m)}$. 
    Then the eigenvalues $\mu_1 \le ... \le \mu_m$ of $\bA$ satisfy $\lambda_k \le \mu_k \le \lambda_{k+n-m}$.
\end{theorem}

With the above, we prove $\kappa(\M_{n,m})$ at best remains constant and at worst grows as the number of lags $n$ increases.

\begin{proposition}\label{prop:Lag_Limit}
    Let $\M_{n,m}$ be as in \eqref{eq:M_general}.
    For $n_1 \le n_2$, $\sigma_{\min}(\M_{n_2,m}) \le \sigma_{\min}(\M_{n_1,m})$ and $\sigma_{\max}(\M_{n_1,m}) \le \sigma_{\max}(\M_{n_2,m})$ so that $\kappa(\M_{n_1,m}) \le \kappa(\M_{n_2,m})$.
\end{proposition}

\begin{proof}
    First, we partition the delay matrix with more lags, $\M_{n_2,m}$, according to the partitioning of Theorem \ref{thm:Cauchy_Interlacing}, so that $\M_{n_1,m}$ is within $\M_{n_2,m}$. 
    Let $\bH = \A_{n_2} \in \C^{m n_2 \times m n_2}$, $\bA = \A_{n_1} \in \C^{m n_1 \times m n_1}$, $\bC = \A_{n_2 - n_1} \in \C^{m (n_2 - n_1) \times m(n_2 - n_1)}$, and $\bB = \bzero \in \C^{m(n_2 - n_1) \times m(n_1)}$, except with the upper-right $\C^{m \times m}$ block being $\bW^*$. 
    Then, by Theorem \ref{thm:Cauchy_Interlacing}, noting that $\lambda_1 = \lambda_{\min}$, $\lambda_n = \lambda_{\max}$, and $\mu_m = \mu_{\max}$,
    \begin{align}\label{eq:interlace1}
        \lambda_{1}(\A_{n_2}) \le \lambda_{1}(\A_{n_1}) \implies \sigma_{\min}(\M_{n_2,m}) \le \sigma_{\min}(\M_{n_1,m}),
    \end{align}
    and likewise,
    \begin{align}\label{eq:interlace2}
        \lambda_{m n_1}(\A_{n_1}) \le \lambda_{m n_2}(\A_{n_2}) \implies \sigma_{\max}(\M_{n_1,m}) \le \sigma_{\max}(\M_{n_2,m})
    \end{align}
    By \eqref{eq:interlace1} and \eqref{eq:interlace2}, we get the bound $\kappa(\M_{n_1,m}) \le \kappa(\M_{n_2,m})$.
\end{proof}

Figure \ref{fig:RecMat_Gen_Bound_Test} simulates Proposition \ref{prop:Lag_Limit} numerically.
For each of the dimensions $m \in \{1,2,4,8\}$ of $\bW$, we tested various lag lengths $n \in \{1,2,4,8\}$.
We partitioned $[0,\frac{1}{2}]$ into $100$ evenly-spaced points for $\sigma_{\max}(\bW)$, and for each $\sigma_{\max}(\bW)$, dimension $m$, and lag $\ell$, we computed the condition number of $\M_{n,m}$ for $100$ randomly-generated $\bW$, generated in the same manner as before.
The figure plots the average condition number and the variance for each $\M_{n,m}$.
The empirical results confirm Proposition \ref{prop:Lag_Limit} in that increasing $m$ or $n$  leads to an increase in $\kappa(\M_{n,m})$, but the result does have a finite upper bound, as shown by the black line.

\section{Discussion}

As a framework for analyzing sequence models, we considered under what conditions linear, autonomous, first-order difference equations of the form \eqref{eq:LRNN}, such as those describing LRNNs, can behave as secondary embeddings of multivariate delay-coordinate maps--maps that use successive lags of input time series to equivalently represent the dynamics of the original system.
For LRNNs, the multivariate delay-coordinate map from the measurement delay-coordinates to the LRNN latent space delay-coordinates is guaranteed to be a stable embedding when the delay matrix $\A_{n,m}$ is full rank and of a low condition number.
For scalar, Hermitian, and general $\bW$, we showed that the \textit{bound} on the singular values and condition numbers of $\A_{n,m}$ is independent of $\dim(\bW) = m$ and the number of lags $n$, at least when the spectrum of $\bW$ lies within the unit circle. 
Although the bound is independent of $n$ and $m$, increasing $n$ and $m$ can only worsen the conditioning of the embedding operator $\M_{n,m}$, but up to a finite limit.
Consequently, $\M_{n,m}$ is a stable embedding when allowing for infinite lookback as $n \to \infty$.
These results provide a theoretical justification for how using LRNNs on time series can generate stable embeddings, and how taking more delays will ensure the embedding conditioning is bounded. 
The results derived on the stability of $\M_{n,m}$ seem to agree with the results with ~\cite{Elaydi_2005, Xu_2015, Jung_2015}, and offers a partial explanation for why the eigenvalues of RNN weights, and particularly RNN weights, tend to train toward a distribution contained approximately within the unit circle ~\cite{Glorot_2010, Orvieto_2023, Jarne_2023, Jarne_2024}. 
In a forthcoming paper, we generalize the framework that RNNs applied to time series consisting of partial observations of dynamical systems are Takens'-type time-delay embeddings from the setting of LRNNs to simple nonlinear Elman RNNs ~\cite{Elman_1990}.

An immediate open question is whether rank-deficient $\M_{n,m}$ can still, with high probability, ensure that LRNNs embed the original dynamics in the available time series into the RNN latent space. 
Takens' embedding theorem states that for a $d$-dimensional attractor, $n \ge 2d+1$ delay-coordinates are needed for the attractor to be properly embedded in the delay-coordinate space.
Requiring $\M_{n,m}: \R^{mn} \to \R^{m(n+1)}$ to be full row rank is a \textit{sufficient} condition for preserving information in the delay-coordinates when mapping to the latent space.
However, since $d < p$, there may be less restrictive conditions for the rank of $\M_{n,m}$--and hence the spectrum of $\bW$--that would still ensure the $d$-dimensional attractor, with high probability, is embedded into the latent space and could be the subject of future exploration.

Another question relates to the fact that, while \eqref{eq:LRNN} models a LRNN, it does not model the training process: does the act of training a LRNN promote a search for more stable and less sensitive delay-coordinate maps?
Given the fundamental nature of machine learning as a training process with train/test splits on data, such an outstanding question is critical for understanding the capabilities of deep learning models to approximate temporal sequences.

\section*{Acknowledgments}

We acknowledge support from the Air Force Office of Scientific Research  (FA9550-24-1-0141).

\section*{Code Availability}

The code used to run the experiments and generate the figures is publicly available on GitHub at \url{https://github.com/fisherng19/LRNNs_Time_Delay_Embeddings.git}

\bibliographystyle{unsrt}
\bibliography{references}

\end{document}